\documentclass[reqno,11pt]{amsart}
\usepackage{amssymb,enumerate,verbatim}
\usepackage{graphicx}
\usepackage{charter}

\usepackage[pdftex,backref,colorlinks,debug]{hyperref}

\makeatletter

\newcommand{\txt}[1]{\text{\rmfamily\mdseries\upshape{#1}}}

\newcommand{\cost}{\rho}
\newcommand{\safe}{\xi}
\newcommand{\growth}{\theta}

\renewcommand{\a}{\mathrm{a}}
\newcommand{\amove}{\vek{a}}
\renewcommand{\d}{\mathrm{d}}
\newcommand{\dmove}{\vek{d}}
\newcommand{\dest}{\mathrm{dest}}

\newcommand{\B}{B}

\newcommand{\Configs}{\mathrm{Configs}}
\newcommand{\D}{D}

\newcommand{\gc}{\mathrm{gc}}

\newcommand{\fail}{\mathrm{fail}}
\newcommand{\halt}{\mathrm{Halt}}

\newcommand{\Histories}{\mathrm{Histories}}
\newcommand{\J}{J}
\newcommand{\M}{M}

\renewcommand{\P}{P}
\newcommand{\cP}{\mathcal{P}}
\newcommand{\Lmoves}{\mathrm{Locmoves}}
\newcommand{\Moves}{\Delta}
\newcommand{\Q}{Q}
\newcommand{\R}{R}
\newcommand{\record}{\alpha}
\newcommand{\Records}{\mathrm{Records}}

\newcommand{\Stacks}{\mathrm{Stacks}} 
\renewcommand{\S}{S}

\newcommand{\U}{U}
\newcommand{\V}{V}

\newcommand{\swell}{\sigma}
\renewcommand{\k}{k}
\renewcommand{\r}{r}
\newcommand{\p}{\vek{p}}

\newcommand{\s}{s}
\newcommand{\success}{\mathrm{succ}}
\newcommand{\start}{\mathrm{start}}
\renewcommand{\t}{t}
\newcommand{\m}{m}
\newcommand{\n}{n}
\renewcommand{\u}{\vek{u}}

\newcommand{\w}{\vek{w}}

\newcommand{\z}{z}

\newcommand{\vek}[1]{\boldsymbol{#1}}

 \theoremstyle{plain}

 \newtheorem{theorem}{Theorem}
 \newtheorem{lemma}{Lemma}[section]

 \newtheorem{Condition}[lemma]{Condition}

\newcommand{\customqed}[1]{{\renewcommand{\qedsymbol}{#1}\qed}}
\newcommand{\varqed}{\customqed{\hbox{$\lrcorner$}}}

 \newcommand\flo[1]{{\lfloor #1\rfloor}}
 
 \newcommand\ul{\underline}

  \renewcommand{\le}{\leqslant}
  \renewcommand{\ge}{\geqslant}

 \newcommand{\df}[1]{{\rmfamily\itshape\mdseries#1}}

 \theoremstyle{definition}

 \newtheorem{Definition}{Definition}[section]
 \newenvironment{definition}{%
   \begin{Definition}}{\varqed\end{Definition}}

 \theoremstyle{remark}

 \newtheorem*{Notation}{Notation}

 \newtheorem{Example}[lemma]{Example}
 \newtheorem{Examples}[lemma]{Examples}

 \newenvironment{example}{%
   \begin{Example}}{\varqed\end{Example}}

 \newtheorem{Remark}[lemma]{Remark}
 \newtheorem{Remarks}[lemma]{Remarks}
 
 \newenvironment{remark}{%
   \begin{Remark}}{\varqed\end{Remark}}
 \newenvironment{remarks}{%
   \begin{Remarks}}{\varqed\end{Remarks}}

 \newcommand{\cB}{\mathcal{B}}
 \newcommand{\cE}{\mathcal{E}}
 \newcommand{\cH}{\mathcal{H}}
 \newcommand{\cK}{\mathcal{K}}
 \newcommand{\cM}{\mathcal{M}}
 
 \newcommand{\cT}{\mathcal{T}}
 \newcommand{\cU}{\mathcal{U}}

 \newcommand{\bbG}{\mathbb{G}}

 \newcommand{\bbZ}{\mathbb{Z}}

\makeatother

\begin{document}

  \title{The angel wins}

  \author{Peter G\'acs}
  \address{Boston University}
  \email{gacs@bu.edu}

  \date{\today}

 \begin{abstract}
The angel-devil game is played on an infinite two-dimensional
``chessboard'' $\bbZ^{2}$.
The squares of the board are all white at the beginning.
The players called angel and devil take turns in their steps.
When it is the devil's turn, he can turn a square black.
The angel always stays on a white square, and 
when it is her turn she can fly at a distance of at most
$\J$ steps (each of which can be
horizontal or vertical) steps to a new white square.
Here $\J$ is a constant.
The devil wins if the angel does not find any more white squares to land
on.
The result of the paper is that if $\J$ is sufficiently large then the
angel has a strategy such that the devil will never capture her.
This deceptively easy-sounding result has been a
conjecture, surprisingly, for about thirty years.  
Several other independent solutions have appeared also in this summer:
see the Wikipedia.
Some of them prove the result for an angel that can make up to two
steps (including diagonal ones).

The solution opens the possibility to solve a number of related problems
and to introduce new, adversarial concepts of connectivity.
 \end{abstract}

\maketitle

\section{Introduction}

The angel-devil game is played on an infinite two-dimensional
``chessboard'' $\bbZ^{2}$.
The squares of the board are all white at the beginning.
The players called angel and devil take turns in their steps.
When it is the devil's turn, he can turn a square black.
The angel always stays on a white square, and 
when it is her turn she can fly at a distance of at most
$\J$ (horizontal or vertical) steps to a new white square.
Here $\J$ is a constant.
The devil wins if the angel does not find any more white squares to land
on.
The result of the paper is that if $\J$ is sufficiently large then the
angel has a strategy such that the devil will never capture her.
This solves a problem that to the authors' knowledge has been open for
about 30 years.

See the bibliography and also the Wikipedia item on the ``angel problem''.

\subsection{Weights}

Let us make the devil a little stronger.
Instead of jumping a distance $\J$ in one step, assume that
the angel makes at most one (vertical or horizontal) 
step at a time, but the devil can deposit only a
weight of size $\swell = 1/\J$ in one step.
The angel is not allowed to step on a square with weight $\ge 1$.
We do not restrict the devil in how he distributes this weight, it need not
be in fractions of size $1/\J$.
 \begin{definition}
Let $\mu(S) = \mu_{t}(S)$
be the weight (measure) of set $S$ at time $t$.
The devil's restriction is 
 \begin{align*}
  \mu_{t+1}(\bbZ^{2})\le\mu_{t}(\bbZ^{2})+\swell. 
 \end{align*}
Let $\cM$ be the set of all measures.
 \end{definition}
 
The main theorem of the paper is the following.
 \begin{theorem}\label{t.angel-wins}
For sufficiently small $\swell$, the angel has a strategy in which
she will never run out of places to land on.
 \end{theorem}

\subsection{Informal idea}

 \begin{definition}
Let 
 \begin{align*}
   \Q > 1
 \end{align*}
 be an integer parameter.
For an integer $\k\ge 0$
we call a square a $\k$-\df{square}, or $\k$-\df{colony}
if the coordinates of its corners are multiples of $\Q^{k}$.
A $(\k+1)$-square can be broken up  into \df{rows}
each of which is the union of $\Q$ disjoint $\k$-squares.
It can also be broken up into \df{columns}.
The side length of a square $\U$ is denoted $|\U|$.
We will call a square $\U$ \df{bad} (for the current measure $\mu$), if 
 \begin{align*}
 \mu(\U)\ge |\U|.
 \end{align*}
Otherwise it is called \df{good}.
 \end{definition}

The angel needs a strategy that works on all scales, to make sure she is
not surrounded in the short term as well as in the long term.
Let us try to develop a strategy for her on the scale of $(k+1)$-colonies,
while taking for granted certain possibilities for her on the scale of
$k$-colonies.
In the context of $(k+1)$-colonies the $k$-colonies will be called
\df{cells}.

A $(\k+1)$-colony $\S^{*}$ can be broken up  into rows $R_{1},\dots,R_{\Q}$.
If $\S^{*}$ is good then we have $\mu(R_{i})< \Q^{k}$ for at least one
$i\in\{1,\dots,\Q\}$.
In this row there is not even a single bad cell,
and at most one cell can be close to badness, the other ones will be
``safe''.
If the row is ``very'' good then even the single square close to badness
will not be spoiled soon.

Suppose that another $(k+1)$-colony $\D^{*}$ is adjacent to $\S^{*}$ on the
north, and that 
 \begin{enumerate}[(a)]
  \item Even the two $(k+1)$-colonies jointly are ``very good'': 
$\mu(\S^{*}\cup \D^{*}) < (1-\delta)Q^{\k+1}$.
Then there is a ``good'' column $C(1)$.
   \item $\S^{*}$ is far from badness: we will call it ``clean''.
It has a clean row $R'(0)$ the angel is in it.
   \item $\D^{*}$ is even farther from badness: we will call it ``safe''.
 \end{enumerate}
Then we may be able to pass from $\S^{*}$ to $\D^{*}$ along column $C(1)$,
passing to it in row $R'(0)$.
By the time we arrive into $\D^{*}$ it may not be safe anymore but
it is still clean, and we can pass into a clean row $R''(1)$ in it.
This simple scheme of passing from one big colony to the next will be
called a \df{step}.

This scheme has many holes yet, and we will fix them one-by-one, adding
new and new complications.
Fortunately the process converges.
Let us look at some of the issues.

 \begin {figure}
\begin {equation*}
\includegraphics{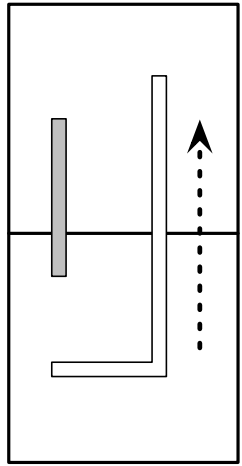}
\end {equation*}
\caption {The white path is the path taken by the angel to get from the
  lower big square to the upper one.
The grey area is to blame that the angel did not go straight.
\label{f.digression}}
 \end {figure}

 \begin{enumerate}[1.]
  \item\label{i.time-bound-idea}
The digression along row $R'(0)$ 
to get into column $C(1)$ causes a delay.
If our delays are not under control (especially in a recursive scheme like
ours) then the devil gains too much time and can put down too much weight.
Let $\U$ be a path of cells (viewed as a union of cells).
We will bound the time along it
by the following expression:
 \begin{align*}
   \tau_{\gc}(\U) + \cost\mu(\U).
 \end{align*}
The term $\tau_{\gc}(\U)$
is essentially the sum of the lengths of the straight runs
of cells in $\U$: we call it the \df{geometric cost}.
Now when the angel makes a digression along row $R'(0)$ to get to column
$C(1)$ it had a reason \emph{outside} its path of cells:
namely, she could not pass straight in column $C(0)$, because of some
weight $\mu(C(0)\setminus\R'(0))$ in this column.
We will upperbound the extra geometric cost by 
$\cost\mu(C(0)\setminus\R'(0))$.
 So the above sum can be estimated as
 \begin{align*}
 \tau_{\gc}(\U^{*}) + \cost\mu(C(0)\setminus\R'(0)) + \cost\mu(\U)
 \le  \tau_{\gc}(\U^{*}) + \cost\mu(\U^{*}).
 \end{align*}
This is how we ``make the devil pay'' for causing a digression: the time
bound formula is conserved when passing from the lower to the higher scale.

  \item Even in a clean row there is possibly one cell that is not safe,
in which we do not want to ``land'' but which we need to ``pass through''.
So we have to look at the situation of three $(k+1)$-colonies on top of
each other, the bottom one, $\S^{*}$ still clean, the top one, $\D^{*}$
still safe, but about the middle one we know only that the union of the
three big colonies together is good.
It is in the nature of the game that we will have to tackle a situation
like this without advance guarantee that we will be able to arrive at the
top.
We have to attempt it and may fail, ending up possibly where we started
(but then knowing that the devil had to spend a lot of his capital on
this).
The strategy will be essentially to try to pass in each column.
Because of the cleanness of $\S^{*}$ there is always a clean row to fall
back to.
This scheme of attempted passage will be called an \df{attack}.
Its implementation uses attacks on the cell level.

  \item In the implementation of attacks, we must be careful about the idea
    of ``falling back to'' a clean row, since the time bound introduced
    above assumes that the path we use is \emph{not self-intersecting}.
The attack must be implemented with some care to achieve this.\footnote{
It is a result in~\cite{Conway96angel} that if the angel has a winning
strategy she also has a self-nonintersecting one.
It is interesting that the strategy we develop is self-nonintersecting for
an apparently different reason.}
We introduce some primitive ``moves'' for the angel
that incorporate some of the possible retreating steps.
Also, to avoid retreat when it is not really needed, we introduce the notion
of a ``continuing attack'', which continues from the result of a failed attack
in the eastern neighbor column without falling back.

  \item With the extra kinds of move a self-nonintersecting implementation 
can be achieved, but there is still a problem.
The delay (extra geometric cost) of the attacks cannot always be
charged directly to some mass outside the path of cells.
Fortunately the case when it cannot is an attack on the level of
$(k+1)$-colonies containing many failed attacks on the level
of $k$-colonies.
A failed attack has a lot of mass \emph{inside it},
and part of it can be used to pay for the extra geometric cost. 
In our new time bound formula therefore if there are $n$
failed attacks they will a contribute a ``profit'' (negative cost):
 \begin{align*}
   \tau(\U) = \tau_{\gc}(\U) + \cost_{1}\mu(\U) - \cost_{2}\n\Q^{k}.
 \end{align*}
This will be sufficient to account for all the digressions.
 \end{enumerate}

\section{Concepts}

Let us proceed to the formal constructions.

\subsection{AD-games}

For our hierarchical solution, we generalize the angel-devil game into a
game called AD-game.

\subsubsection{Parameters}

 \begin{definition}
A union $\U = \U_{1}\cup\dots\cup\U_{n}$ of (horizontally or
vertically) adjacent squares of equal size will be called a \df{run}.
We write 
 \[
   |\U| = n |\U_{1}|.
 \]
 \end{definition}

 \begin{definition}
Let
 \begin{align}\label{e.delta-safe-swell}
      2/3 < \safe < 1, 
\quad 0 < \delta < \safe/2,
\quad 0 < \swell,
\quad \cost_{1}>\cost_{2}>0
 \end{align}
be real parameters to be fixed later.
We say that the run $\U = \U_{1}\cup\dots\cup\U_{n}$ 
is \df{bad} for measure $\mu$ 
if $\mu(\U)\ge |\U_{1}|$, otherwise it is good.
We will say that it is $i$-\df{good}
if 
 \begin{align*}
   \mu(\U) < (1-i\delta) |\U_{1}|.
 \end{align*}
In particular, $0$-good means simply \df{good}.
Similarly, we will call a run $i$-\df{safe} if 
$\mu(\U) < (\safe-i\delta) |\U_{1}|$.
 \end{definition}

 \begin{definition}
For an integer $\B>0$ and an integer $x$ we write
 \begin{align*}
   \flo{x}_{\B} = \B\cdot \flo{x / \B}.
 \end{align*}
Similarly for a vector $\u = (x,y)$ with integer coordinates,
 \begin{align*}
   \flo{\u}_{\B} = \B\cdot (\flo{x/\B}, \flo{y/\B}).
 \end{align*}
The set
 \begin{align*}
  \P = \{(0,1), (0,-1), (1,0), (-1,0)\}.
 \end{align*}
will be called the set of \df{directions}.
These directions will also be called \df{east, west, north, south}.
 \end{definition}

\subsubsection{The structure}

 \begin{definition}
An \df{AD-game} $\bbG$ with \df{colony size} $\B$
consists of steps alternating between a player called angel and
another one called devil.
At any one time, the current configuration determines the possibilities
open for the player whose turn it is.
We will only consider strategies of the angel.
More precisely, a game
 \begin{align}\label{e.gener-game}
 \bbG = \bbG(\B, \Moves_{\a}, \Moves_{\d}) 
 \end{align}
is defined as follows.
As before, the devil controls a measure $\mu_{t}$
and can add the amount
$\swell>0$ to the total mass at each time, so that $\mu_{\t+1}\ge \mu_{\t}$
and $\mu_{t+1}(\bbZ^{2})-\mu_{t}(\bbZ^{2}) \le\swell$.

The plane is partitioned into a lattice of squares
of size $\B$ called \df{colonies}.
Point $\u$ is contained in the colony
 \begin{align*}
 \cB(\u) = \flo{\u}_{\B} + \{0,\dots,\B-1\}^{2}.
 \end{align*}
  \end{definition}

The game is played using the following definitions.

 \begin{definition}
The game follows a sequence of moves $\r=1,2,\dots$, associated with an
increasing sequence of integer times $\t_{\r}$.
At move $\r$, times $t_{1},\dots,t_{\r}$ are already defined,
the angel stays at a point $\p_{\r} \in\cB(\p_{r})$.
The triple
 \begin{align*}
   (\t_{\r}, \p_{\r}, \mu_{\t_{\r}})
 \end{align*}
will be called an \df{essential configuration} of the game in step $\r$.
Let us call \df{default essential configuration}
the configuration $(0,(0,0),0)$, that is
the configuration at time 0, position 
at the origin $(0,0)$, the null measure.

The game starts at time $\t_{0}=0$ from position $\p_{0}=(0,0)$ 
with initial measure $\mu_{0}=0$.
The position
 \begin{align*}
   \w_{\r} = \flo{\p_{\r}/\B}
 \end{align*}
will be called the \df{colony position} at step $\r$.
 \end{definition}

\subsubsection{Moves and the angel's constraints}

Let us see what are the potential moves.

  \begin{definition}
There is a finite set
  \begin{align*}
    \Pi
  \end{align*}
 of symbols called \df{potential moves}.
Each move $\z$ has finite sets 
  \begin{align*}
    \cE(\z)\subset \cH(\z) \subset \bbZ^{2}
  \end{align*}
where $\cH(\z)$ is called the \df{template} of the move,
and $\cE$ is called the set of \df{end positions} in the template.
There is also an  element $\dest(z)\in\cE(\z)$
called the \df{destination position} of $\z$.
  \end{definition}

Let us see how moves will be used before defining them.

 \begin{definition}
At any time $\t_{\r}$, when staying in some \df{start colony} $\S$,
the angel chooses a \df{move} $\z=\z_{r}$ from the set 
 \begin{align*}
  \Moves_{\a}(\p_{\r},\mu_{\t_{r}})\subset\Pi.
 \end{align*}
She loses if this set is empty.
Let us call $\t_{\r}$ the \df{start time} of the move and $\t_{\r+1}$ the
\df{end time} (unknown yet) of the move.
The \df{body} of the move is the set
 \begin{align*}
  \M = \M(\w,\z) = \S+ \B\cdot\cH(z).
 \end{align*}
Colonies $\w+\u$ for $\u\in\cE(\z)$ are the possible \df{end
colonies} of the move.
There will be several more restrictions on the moves the angel can choose.
The devil will deposit the angel at a point in a certain end colony,
at the endtime of the move chosen by him.
There will be several restrictions on the devil's choice of place and time.
 \end{definition}

 \begin{definition}
A pair 
 \begin{align*}
  \amove=(\w,\z)
 \end{align*}
where $\w\in\bbZ^{2}$ and $\z\in\Pi$ is called a \df{located move}.
We call the colony $\cB(\w)$
where the angel stays at the beginning of the move
the \df{starting colony} for this move.
A \df{default located move} is a default move starting from the origin
$(0,0)$.
 \end{definition}

We will see that the devil's answer can end
certain moves in two different ways: in
``success'' or in ``failure''.
The intuitive meaning of failure is the failure to get through some
``obstacle''.
A successful move always ends in its destination colony.
A failed attack will end in one of its end colonies (possibly the
destination colony).
Formally, the definition of allowed moves, success and failure uses the
notion of a point ``clear'' for starting a certain kind of move
and a point ``clear'' for ending a certain kind of move in success or
failure.
These notions will be defined below.

Here is a description of the different kinds of move.
They are also illustrated in Figure~\ref{f.moves}.
These details will be motivated better once
it is understood how they are used.

 \begin {figure}
\begin {equation*}
\includegraphics{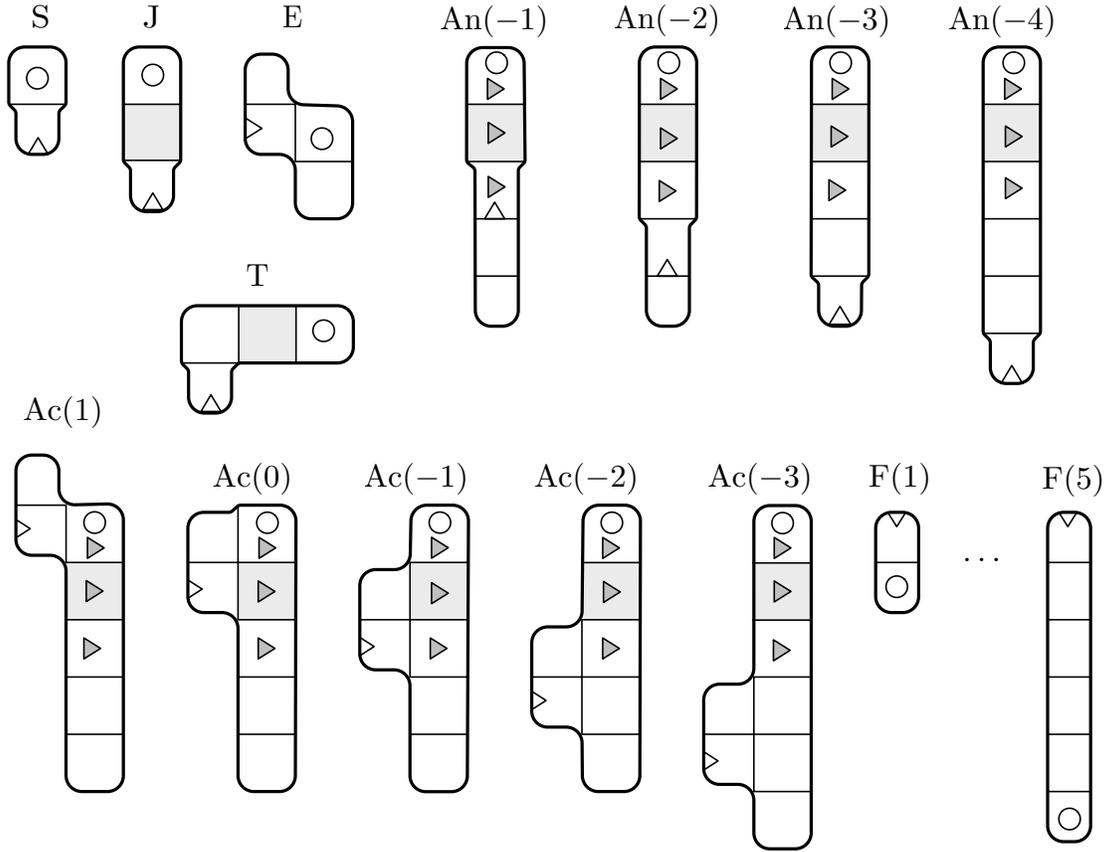}
\end {equation*}
\caption {Possible moves.
The template of each move is a union of squares.
Grey squares are the obstacle positions of jumps and attacks.
A white triangle marks the starting position.
White circle marks the destination position,
grey triangles the transit positions, showing the direction of the
sweep.
The level of a continuing attack is the difference between
the height of the starting position and the height of the obstacle.
Notation: S: step,  J: jump, T: turn, E: escape, F: finish, A: attack, n:
new, c: continuing, $(i)$: level or length $i$.
\label{f.moves}}
 \end {figure}

 \begin{definition}\label{d.move-types}
Moves are of the following kinds.
Each move has a \df{starting direction} and \df{landing direction}.
Only the turn move has a landing direction different from the starting
direction, for all other moves, their \df{direction} is both the starting
and landing direction.
Attacks and escapes also have a \df{passing direction}: so we talk about
a northward attack or escape passing to the east.

We also distinguish \df{new} moves and \df{continuing} moves.
The step, jump, turn and new attack are new moves, the continuing attack,
escape and finish are continuing moves.

 \begin{enumerate}[$\bullet$]

  \item 
A northward  \df{step} has $\cH=\{(0,0),(0,1)\}$, $\dest = (0,1)$.

  \item
A northward \df{jump} has $\cH=\{(0,0),(0,1), (0,2)\}$, $\dest = (0,2)$.
The position $(0,1)$ of a northward jump is called its \df{obstacle
  position}.
Other jumps are obtained if we rotate a northward step 
by multiples of 90 degrees.

  \item
A \df{northward escape continuing to the east} has
$\cH=\{(0,0),(0,1),(1,0),(1,-1)\}$, $\dest = (1,0)$.
Other escapes are obtained by reflection and rotations.

  \item
A \df{southward finish} of length $1\le i \le 5$ has $\dest = (0,-i)$,
$\cH=\{(0,0),\dots,(0,-i)\}$.
Other finishes are obtained by rotations.

  \item Let us describe the kind of moves called \df{northward attacks
    continuing to the east}.
Other attacks are obtained using reflexion and rotations.
It is more convenient to describe a \df{pre-template} $\cH'$
and a starting position $\p'$ from which the
real template $\cH$ is obtained by subtracting the starting position:
$\cH=\cH'-\p'$.
Of course, the starting position, destination position, and so on are also
shifted when going from $\cH'$ to $\cH$.
The pre-template is a superset of 
 \begin{align*}
   \cU = \{(0,-3),(0,-2),(0,-1),(0,0),(0,1)\},
 \end{align*}
with the obstacle in $(0,0)$, and the destination in $(0,1)$.

All attacks have destination position $(0,1)$ in $\cH'$.
They also have a set $\cT\subset\cE$ of \df{transit positions}: in
$\cH'$ these are $(0,-1),(0,0),(0,1)$.
Each attack has an integer \df{level} $\s$ with $-4\le\s\le 1$.

Attacks are also divided into \df{new} or \df{continuing} attacks.
New attacks have level $-4\le\s\le -1$, the
pre-template is equal to $\cU \cup (0,\s)$, and its start position is
$\p'=(0,\s)$.
For continuing attacks, $-3\le\s\le 1$,
$\cH' = \cU\cup\{(-1,\s),(-1,\s+1),(0,-\s)\}$, and
the start position is $\p'=(-1,\s)$.

  \item A nortward \df{turn} is a combination of a
northward step and an eastward or westward jump or new attack
(the latter is northward-continuing).
The direction of the jump or attack is called the \df{landing direction} of
the turn.

 \end{enumerate}
The body $\M$ of a 
northward continuing attack or escape consists of left and right columns.
The right column will be called the \df{reduced body} and denoted
 \begin{align*}
   \ul\M.
 \end{align*}
 \end{definition}

 \begin {figure}
\begin {equation*}
\includegraphics{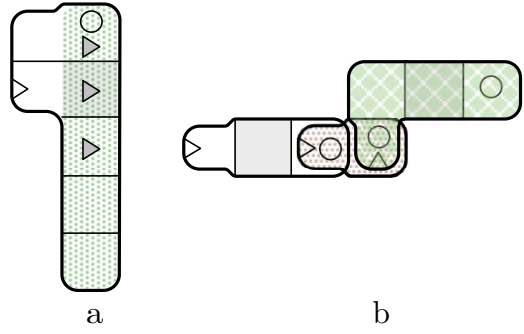}
\end {equation*}
\caption {a: The reduced body of a continuing attack is shaded.
b: An eastward jump followed by an eastward step followed by a
northward-eastward turn.
\label{f.jump-step-turn}}
 \end {figure}

\begin{definition}
The definition of $\Moves_{\a}$ and $\Moves_{\d}$ uses
the relations $\cK_{\start}$ and $\cK_{\fail}$.
These are defined as follows, for measure $\mu$.
We have
 \begin{align*}
     \cK_{\start} &\subset \cM\times\bbZ^{2}\times\P,
\\   \cK_{\fail}  &\subset \cM\times\bbZ^{2}\times\bbZ^{2}\times\Pi.
 \end{align*}
If $(\mu,\p,x)\in\cK_{\start}$ then point $\p$ is \df{clear} for measure
$\mu$ start a move in direction $x$.
If $(\mu,\p,\w,\z)\in\cK_{\fail}$ then point $\p$ is \df{clear} 
for measure $\mu$ to end the located move $(\w,z)$ in failure.
\end{definition}

 \begin{definition}
Let us specify the set $\Moves_{\a}(\p,\mu)\subset\Pi$ 
of possible moves for the angel when she is at position $\p$ with measure $\mu$.

We will only consider moves in the northward direction.
The requirements for other directions are obtained using rotation and
reflection.
 \begin{enumerate}[(a)]

  \item A new move $\z$ with starting direction $x$ is allowed only from a
    point $\p$ that is clear in direction $x$, that is we must have
$(\mu,\p,x)\in\cK_{\start}$.

        A continuing move $\z$ is allowed from a point $\p$ only if
$\p$ is clear for some failed located move $(\w,\z')$
having the same landing direction and passing direction: that is,
$(\mu,\p,\w,\z')\in\cK_{\fail}$.

  \item The weight of the body is at most $3\B$ (this bound is not
    important, just convenient).

  \item The destination colony is $(-1)$-safe.

  \item If the move is a step then the body is $(-1)$-safe.

  \item If the move is a jump then the body is $1/2$-good.

  \item If the move is a northward escape then its reduced body
(its right column) is $(-1)$-safe.

  \item If the move is an attack then:
   \begin{enumerate}[(A)]

     \item The run in the reduced body below the obstacle
       colony is $(-1)$-safe.

     \item If it is a new attack, the 
body is good; if it is a continuing attack, the reduced body is good.

   \end{enumerate}
  \item If the move is a turn then it satisfies the conditions of
its constituent step and (jump or attack).

 \end{enumerate}
Let the \df{default move} be an eastward step.
\end{definition}

 \begin {figure}
\begin {equation*}
\includegraphics{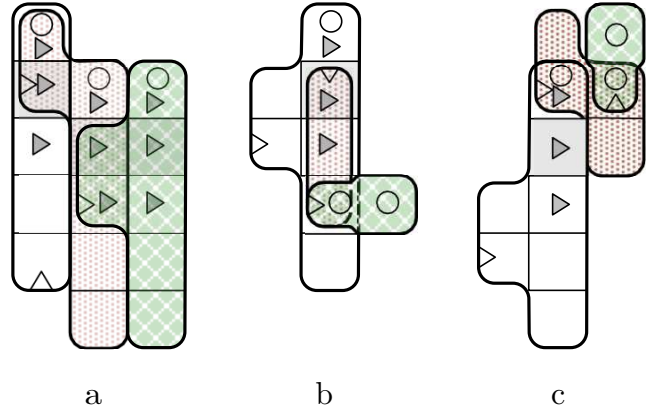}
\end {equation*}
\caption {How attacks can be continued.
a: A new attack fails on level 0 (on the level of its obstacle).  
Is followed by a continuing attack which
picks up there (as shown by the first pair of grey and white triangles).
This attack is of level $1$ since the obstacle in the next column is one 
unit below the transit location where the failure occurred.
The second attack fails on level $-1$, and is followed by a continuing
attack of level $-1$ since the next obstacle is on the same level as the
previous one.
b: A continuing northward 
attack fails on level 0, and is followed by a southward finish move
and an eastward step.
c: A continuing attack fails on level 1, and is followed by an escape move
and a northward step.
\label{f.continue}}
 \end {figure}

\subsubsection{Paths and the devil's constraints}

Before giving the devil's constraints, some more definitions are needed.

 \begin{definition}
A \df{configuration} is a tuple
 \begin{align*}
   (\t, \p_{\r}, \mu, j).
 \end{align*}
Here, $(\t,\p,\mu)$ is an essential configuration, and the symbol 
$j\in\{\success,\fail\}$
shows whether the previous move of the angel \df{suceeded} or \df{failed}.
The \df{default configuration} 
consists of the default configuration with
$j=\success$ added.
Let 
 \begin{align*}
   \Lmoves(\B),\Configs
 \end{align*}
be the sets of located moves and configurations respectively.
A sequence $(\amove_{1},\dots,\amove_{\m})$ of located moves
with $\amove_{i}=(\w_{i},\z_{i})$
is called a \df{path} if $\w_{i+1}-\w_{i}\in\cE(\z_{i})$ holds
for all $i<\m$.
\end{definition}

Simple paths will be defined to be essentially self-nonintersecting.
But the finish move and some steps following it in the same direction
can overlap with the 
failed attack before it, so the following definition takes this into
account.

 \begin{definition}
Take a path, and let $\M_{1},\dots,\M_{\m}$ be all the bodies of its
located moves \emph{except the finish moves}.
Let $\S_{i}$ be the starting colony of the move of $\M_{i}$.
The path is \df{simple} if for all $i$ we have
$\M_{i+1}\cap\bigcup_{j=1}^{i}\M_{j}=\S_{i+1}$.
 \end{definition}

 \begin{definition}
If the angel's move is $\z_{\r}$, the set of possible moves of the devil is
 \begin{align*}
 \Moves_{\d} = \Moves_{\d}(\p_{\r},\mu_{\t_{\r}},\z_{\r}) \subseteq \Configs.
 \end{align*}
The devil's move is the next configuration
$\dmove_{\r+1} = (\t_{\r+1},\p_{\r+1},\mu_{\t_{\r+1}},j_{\r+1})$,
but his choice is restricted.
First, of course $\t_{\r} \le \t_{\r+1}$, further
$0 \le \mu_{\t_{\r+1}} - \mu_{\t_{r}} \le \swell(\t_{\r+1}-t_{\r})$. 
The other restrictions defining $\Moves_{\d}$ can be divided into
\df{spatial} and \df{temporal} restrictions.
Let us give the spatial restrictions first, assuming
that the devil deposits the angel at a point $\p$ with measure $\mu$.
  \begin{enumerate}[(a)]

  \item
In case of success, $\p$ is clear in the landing direction $x$ of $\z$,
that is $(\mu,\p,x)\in\cK_{\start}$.

In case of failure the located move $(\w,\z)$ was a continuing move and
$\p$ is clear for failure for this move, that is
$(\mu,\p,\w,\z)\in\cK_{\fail}$.

  \item If a new northward attack is not successful
then the northward run of the body beginning from the starting colony
is bad at its end time (of course, the same goes for other directions).

   If a continuing attack is not successful
then the reduced body is bad at its end time.
 \end{enumerate}
\end{definition}

Before giving the devil's temporal restrictions, some more notions
concerning histories are needed.

\begin{definition}
Given a path $(\amove_{1},\dots,\amove_{\m})$ with
$\amove_{i}=(\w_{i},\z_{i})$, and a sequence
$(\dmove_{1},\dots,\dmove_{\m+1})$
of configurations $\dmove_{i} = (\t_{i},\p_{i},\mu_{i},j_{i})$,
the sequence
 \begin{align*}
   (\dmove_{1},\amove_{1},\dmove_{2},\amove_{2},
   \dots,\dmove_{\m}\amove_{\m},\dots,\dmove_{\m+1})
 \end{align*}
will be called a d-\df{history} if its configurations
obey the spatial restrictions and the restrictions on $\mu_{t}$ for the devil
given above.

Under the same restrictions, the sequence
$(\dmove_{1},\amove_{1},\dots,\dmove_{\m},\amove_{\m})$
is called an a-\df{history}.
The set of all a-histories or d-histories will be denoted by
 \begin{align*}
   \Histories_{\a}(\B), \Histories_{\d}(\B).
 \end{align*}
The default d-history is $(\dmove_{0})$, consisting of a single default
configuration $\dmove_{0}$. 
If $\chi$ is a history then let 
 \begin{align*}
 \amove(\chi)
 \end{align*}
be the path consisting of the angel's moves in it.
\end{definition}

\begin{definition}
We will use the addition notation, for example
 \begin{align*}
   (\amove_{1},\dmove_{1},\dots,\dmove_{\m},\amove_{\m}) = 
   (\amove_{1},\dmove_{1},\dots,\dmove_{\m}) + \amove_{\m}.
 \end{align*}
For another addition notation, if $\chi =
(\dmove_{1},\amove_{1},\dots,\dmove_{i},\dots,\amove_{\m},\dmove_{\m+1})$ 
is a d-history then we can write
$\chi=\chi_{1}+\chi_{2}$ where
$\chi_{1} = (\dmove_{1},\dots,\amove_{i-1},\dmove_{i})$, 
$\chi_{2} = (\dmove_{i},\dots,\amove_{\m},\dmove_{\m+1})$
the d-histories from which it is composed.
A d-history $(\dmove,\amove,\dmove')$ will be called a \df{unit}
d-\df{history}. 
Thus, every d-history is the sum of unit d-histories.

Similarly, if $\chi = (\dmove_{1},\amove_{1},\dots,\dmove_{i},\dots,\amove_{\m})$
is an a-history then we can write
$\chi=\chi_{1}+\chi_{2}$ where
$\chi_{1} = (\dmove_{1},\dots,\amove_{i-1})$, 
$\chi_{2} = (\dmove_{i},\dots,\amove_{\m})$,
the a-histories from which it is composed.
 \end{definition}

 \begin{definition}
An 1-step a-history $(\dmove,\amove)$ will be called a \df{record}.
Thus, every a-history is the sum of records.
Let 
 \begin{align*}
   \Records(\B) = \Configs\times\Lmoves(\B)
 \end{align*}
denote the set of all possible records.
The \df{default record} has the form
$\record_{0} = (\dmove_{0},\amove_{0})$
consisting of 
the default configuration and the default move (eastward step).
 \end{definition}

 \begin{definition}\label{d.time-bound}
We define a \df{time bound} $\tau(\chi)$ for a history $\chi$.
Let $\mu$ be the measure in the last configuration, let
$\U$ be the union of the bodies of all located moves in the path
$\cP=\amove(\chi)$ and let 
$\n$ be the number of failed continuing attacks in $\chi$.
Then
 \begin{align*}
   \tau(\chi) = \cost_{1}\mu(\U) - \cost_{2}\n\B + \tau_{\gc}(\chi)
 \end{align*}
where $\tau_{\gc}(\chi)$ is the called the \df{geometric cost}, or
the \df{geometric component} of the time bound,
which we will define now.
Let $\chi$ be a unit history containing a northward move that is not a
turn, and let 
$y_{\r}$ and $y_{\s}$ be the $y$ coordinates of the starting point and
the endpoint respectively.
We define the \df{geometric cost} of $\chi$ as 
$\tau_{\gc}(\chi) = y_{\s}-y_{\r}$.   
(For an attack this can be negative.)
For moves in other directions, the geometric
cost is obtained by rotation accordingly.
If $\chi$ is a single turn then $\tau_{\gc}(\chi)=8\B$.
If $\chi$ is an arbitrary history then
let us decompose it into a sequence of unit histories $\chi_{i}$ and
define $\tau_{\gc}(\chi)=\sum_{i} \tau_{\gc}(\chi_{i})$.

Now, the \df{temporal restriction} of the devil is the following:
the time of any d-history with a simple path is 
bounded by its time bound defined above.

Let us call an a-history or d-history \df{legal} if both the angel's and
the devil's moves in it are permitted in the game, based on the sequence of
preceding elements.
 \end{definition}

 \begin{remarks}\
 \begin{enumerate}[1.]
  \item
Though the time bound contains negative terms, it never becomes
negative because of $\cost_{2}<\cost_{1}$: it can even be lowerbounded by
a constant times the number of moves in the history.
  \item
Why do we ``profit'' only from failed
continuing attacks and not from failed new attacks?
This will be understood later, when we implement a scaled-up attack.
A new attack typically needs some preparation steps that must be charged
against its mass.
 \end{enumerate}
 \end{remarks}

In terms of histories, game $\bbG$ can be described as follows.
It is started from some initial configuration $\dmove_{1}$, with
$\mu=0$, say in the middle of the cell $(0,0)$.
Now the angel adds her located move $\amove_{1}$ to the history.
The devil follows with the the next configuration 
$\dmove_{2}$, and so on.
So the angel's moves can be viewed as located moves, the devil's
moves as configurations.
Of course, each of these obeys the constraints given above.

 \begin{example}
In order to specify a simple example of an AD-game we must specify
the parameters $\Q,\safe,\delta,\cost_{i}$ 
and the relations $\cK_{\start},\cK_{\fail}$. 
Let $\Q=1$ and let us fix the other
parameters in any way obeying the above restrictions
(more restrictions come later to make scale-ups possible).

Let $(\mu,\p,\z)\in\cK_{\start}$ if $\p$ is $(-1)$-safe for $\mu$.
In case of a northward attack $\z$ let
$(\mu,\p,\z,\w)\in\cK_{\fail}$ if in the body of located move $(\w,\z)$,
the point $\p$ is below the obstacle colony of the attack.

It is not hard to see that this game is
essentially equivalent to the original angel-devil game.
 \end{example}

\subsection{Scaling up}\label{ss.scale-up}

Let us define some concepts of cleanness for runs.

 \begin{definition}\label{d.clean-run}
In a run $\U$ of colonies let us call the \df{obstacle} an element
with largest weight (say the first one).
A run will be called $i$-\df{step-clean} if 
every run of two consecutive colonies in $\U$ is $i$-safe.
A run $\U$ of colonies of game $\bbG$ is $i$-\df{unimodal} for some integer
$i$ if the runs on both sides of the obstacle are $i$-step-clean.
It will be called $i$-\df{clean} if it is $i$-unimodal and
every run of three consecutive colonies in $\U$ is $(i+1)$-good.

A run will be called \df{clean}, and so on if it is 0-clean, and so on.

Let $\U_{1},\dots,\U_{\n}$ be a clean run, and let 
$1=n_{1}<n_{2}<\dots<n_{\m}=\n$ be a sequence of indices 
with $n_{i+1}\le n_{i}+2$ such that
$\U_{n_{i}}$ is safe, and if $n_{i+1}=n_{i}+1$ then also
$\U_{n_{i}}\cup\U_{n_{i+1}}$ is safe.
This sequence will be called a \df{walk}: it consists of steps
and jumps that can be carried out.
 \end{definition}

Now we are ready to define scaled-up games.

 \begin{definition}
Using the integer parameter $\Q$ introduced above, let
 \begin{align*}
 \B^{*}=\Q\B.
 \end{align*}
For clarity we will generally denote by $\record^{*}$ the elements of
$\Records(\B^{*})$, and use the ${}^{*}$ notation similarly also in other
instances where it cannot lead to confusion.
Or, if $\record$ denotes an element of $\Records(\B^{*})$ we might denote
by $\record_{*}$ a record of $\Records(\B)$.

Each colony $\U^{*}$ of size $\B^{*}$ is the union of $\Q$ 
\df{rows} of colonies of size $\B$, and also the union of $\Q$ \df{columns}
of them.
Let $\U^{*}$ consist of colonies $\U_{ij}$ ($1\le i,j\le \Q$) of size $\B$.
In the context of game $\bbG^{*}$ the latter will be called \df{small
  colonies}, or \df{cells}.

Given a game $\bbG$ of colony size $\B$, the game
 \begin{align*}
 \bbG^{*} = \bbG(\B^{*}, \Moves^{*}_{\a}, \Moves^{*}_{\d})
 \end{align*}
will be defined similarly to game $\bbG$, but with colony size $\B^{*}$,
except that the sets $\cK^{*}_{\start}$,
$\cK^{*}_{\fail}$ are defined as a function of the corresponding sets in 
$\bbG$, as given below.
\end{definition}

 \begin {figure}
\begin {equation*}
\includegraphics{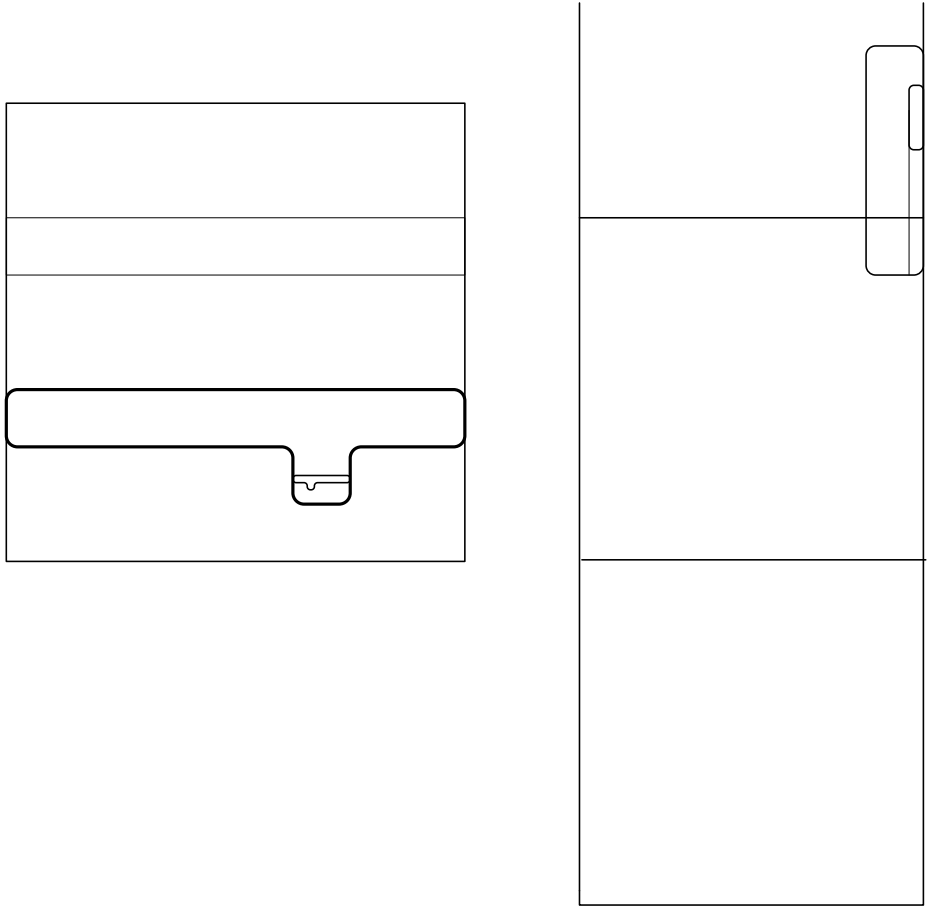}
\end {equation*}
\caption {A northward clear point and a point that is
clear for failure in a northward attack passing to the east.
\label{f.clear}}
 \end {figure}

 \begin{definition}
Let a point $\p$ be in a cell $\U$ within the big colony $\U^{*}$,
and let $\mu$ be the current measure.

We will say that a point $\p$ is \df{clear} in direction $x$
with respect to measure $\mu$ in game $\bbG^{*}$, that is 
$(\mu,\p,x)\in\cK^{*}_{\start}$ if (assuming without loss of generality
that $x=(0,1)$, that is northward):
 \begin{enumerate}[(1)]
  \item It is northward clear for $\mu$ in game $\bbG$, that is 
$(\mu,\p,x)\in\cK_{\start}$.
  \item Colony $\U^{*}$ is $(-2)$-safe.
  \item There are at least $(\kappa-2)$ clean rows in $\U^{*}$ north of $\U$,
and the first one is reachable from it in one (allowed) northwards step.
 \end{enumerate}

Let $(\w,\z)$ be a located move with northward landing direction
and eastward passing direction in $\bbG^{*}$ such that $\U^{*}$ is one of
the end colonies in the body $\M^{*}$ of $(\w,\z)$.
We will say that a point $\p$ is clear for failure of $(\w,\z)$ in
$\bbG^{*}$ if one of the following conditions is satisfied.
 \begin{enumerate}[(1)]

  \item In the game $\bbG$, 
it is eastward clear (that is, $(\mu,\p,(1,0))\in\cK_{\start}$)
and the column of cells in the body $\M^{*}$ 
south of the cell $\U$ (including $\U$) is $(1/2)$-step-clean.

  \item There is a northward continuing located move $(\w',\z')$ passing to
the east such that $\p$ is clear for failure for $(\w',\z')$,
(that is $(\mu,\p,\w',\z')\in\cK_{\fail}$), 
and the column of cells in the body $\M^{*}$ south of the cell $\U$ 
(not including $\U$) is $(1/2)$-step-clean.

 \end{enumerate}
 \end{definition}

Intuitively, a point is northward clear if starting from it,
we have some freedom to choose in which column to move northward.
A point is clear for failure in a northward continuing attack passing to
the east if it is on the east edge of the attack body, and a southward
``retreat'' is possible from it.

We are interested in translating
moves of the angel in $\bbG^{*}$ into sequences of moves in $\bbG$.
Recall that a record is a one-step a-history.

 \begin{definition}\label{d.phi}
Consider the pair of functions
 \begin{align*}
   \phi : \Histories_{\d}(\B)\times\Records(\B^{*}) &\to \Pi\cup\{\halt\},
\\  (\chi,\dmove,\amove) &\mapsto \phi(\chi \mid \dmove,\amove),
\\   \J : \Histories_{\d}(\B)\times\Records(\B^{*}) &\to \{\success,\fail\},
\\   (\chi,\dmove,\amove) &\mapsto \J(\chi \mid \dmove,\amove).
 \end{align*}
To simplify writing if $(\t',\p',\mu',j')$ is the last element of
$\chi$ with $\w'=\flo{\p}_{\B}$ then we will write
 \begin{align*}
     \hat\phi(\chi\mid\dmove,\amove) = (\w',\phi(\chi \mid \dmove,\amove)),
\quad   \hat\J(\chi\mid\dmove,\amove) = (\t',\p',\mu',\J(\chi\mid \dmove,\amove)).
 \end{align*}
We say that $(\phi,\J)$ is 
an \df{implementation map} from game $\bbG$ to game $\bbG^{*}$
if it has the following properties, whenever $\chi$ is a legal d-history:
 \begin{enumerate}[(a)]
  \item
If $\phi(\chi \mid \dmove,\amove) \ne \halt$ then 
it is a permitted move of the 
angel in game $\bbG$ following any history ending in $\chi$.
  \item\label{i.when-halt}
If $\phi(\chi \mid \dmove,\amove) = \halt$ then 
$\chi$ consists of at least two steps, and
$\hat\J(\chi\mid\dmove,\amove)$ is a permitted move of 
the devil in the game $\bbG^{*}$ following any history ending in
$(\dmove,\amove)$.
  \end{enumerate}
 \end{definition}

The implementation can be viewed in the following way.
The angel and devil of game $\bbG^{*}$ play while pursuing
a parallel game of $\bbG$ as follows.
When it is the angel's turn in $\bbG^{*}$
she will translate her move
into a \df{local strategy} of $\bbG$, a strategy that can also halt.
Step $\r$ of game $\bbG^{*}$ corresponds to step $\s_{\r}$ of
game $\bbG$, but at the essential configuration as in game $\bbG^{*}$.
A one-step d-history
$(\dmove_{\r},\amove_{\r},\dmove_{\r+1})$ of game $\bbG^{*}$
will correspond to a longer d-history
 \begin{align*}
   (\dmove_{*\s_{\r}},\amove_{*\s_{\r}},\dmove_{*(\s_{\r}+1)},\dots,
   \amove_{*\s_{\r+1}-1},\dmove_{*\s_{\r+1}})
 \end{align*}
of game $\bbG$ generated as follows.
Between steps $\s_{\r}$ and $\s_{\r+1}$ of game $\bbG$, 
the angel will use the following strategy.
Let 
 \begin{align*}
 \chi_{*i}=(\dmove_{*\s_{\r}},\dots,\amove_{*(i-1)}\dmove_{*i})
 \end{align*}
be the d-history in the game $\bbG$ since the
start of the implementation of the current move of $\bbG^{*}$,
where $\dmove_{*i}=(\t_{*i},\p_{*i},\mu_{*i},j_{*i})$.
Then the next move $\amove_{*i}$ in game $\bbG$ is computed as follows,
with $\record_{\r} = (\dmove_{\r},\amove_{\r})$:
 \begin{align*}
  \amove_{*i} = \hat\phi(\chi_{*i} \mid \record_{\r}).
 \end{align*}
From here, the local d-history is extended as
$\chi'_{*}=\chi_{*i}+\amove_{*i}$. 
The devil of $\bbG^{*}$ allows the devil of
$\bbG$ to play and generate
the d-history $\chi_{*{i+1}}=\chi'_{*} + \dmove_{*(i+1)}$.
Step number $\s_{\r+1}$ will be reached in $\bbG$
when the angel in game $\bbG$ chooses the symbol $\halt$.
At this point the devil of game $\bbG^{*}$
chooses the new configuration of game $\bbG^{*}$, namely
 \begin{align*}
   \dmove_{\r+1} = \hat\J(\chi_{*}\mid\record_{\r}).
 \end{align*}
This next configuration of game $\bbG^{*}$ is almost the same as the
last configuration of the subgame:
the essential configurations are the same, 
only the question whether the last move of $\bbG^{*}$
(the one we have just implemented) is 
successful will be decided using the function $\J(\cdot)$.

Thus to any d-history $\chi$
of $\bbG^{*}$ corresponds some d-history $\chi_{*}$ of game
$\bbG$ that shares with $\chi$ the essential initial and final
configurations.
The correspondence assigns disjoint subhistories of $\chi_{*}$ to each unit
history of $\chi$.
Of course, $\chi_{*}$ is not a function of $\chi$ but the map $(\phi,\J)$ 
translates any
strategy that the angel has for $\bbG^{*}$ into a strategy for $\bbG$.

 \begin{definition}
An \df{amplifier} consists of a sequence of games $\bbG_{1}$,$\bbG_{2}$,
$\dots$ where $\bbG_{\k+1}=\bbG_{\k}^{*}$
and implementation maps $(\phi_{\k},\J_{\k})$ from
$\bbG_{\k}$ to $\bbG_{\k+1}$.
 \end{definition}

In the amplifier built in the present paper 
the maps $(\phi_{\k},\J_{\k})$ will not depend on $\k$ in any significant
way: only the scale changes.

\subsection{The main lemma}\label{ss.main-lemma}

Before stating the 
main lemma, from which the theorem will follow easily, let us constrain our
parameters.

Let
 \begin{align}
    \label{e.nu-kappa}
 \nu = 17\Q,
\quad \kappa \ge 12
 \end{align}
be an integer parameters.
The parameter $\nu$ will serve as the upper bound on the number of small
moves in an implementation of a big move.
The parameter $\kappa$ will have the approximate role that in a safe colony
there will be at least $\kappa$ guaranteed good rows.
Below, the parameter $\growth$ has the meaning that the time taken by
a single move will be upperbounded by $\growth\B$.
Let
 \begin{align}
  \label{e.Q-lb}
                \Q &> 2\kappa/(1-\safe),
\\\label{e.cost2}
         \cost_{2} &= 8,
\\\label{e.cost1}
         \cost_{1} &> 22\Q/(1-\safe),
\\\label{e.growth}
           \growth &= 2(6 + 3\cost_{1}),
\\ \label{e.delta-ub}
     \delta &< \min((1-\safe)/6, (\safe-2/3)/\Q),
\\ \label{e.swell-ub}
     \swell &< \min(\delta/(3\nu\growth), 1/(2\cost_{1})).
 \end{align}
These inequalities can be satisfied by first choosing $\Q$ to
satisfy~\eqref{e.Q-lb}, then $\cost_{1}$ to satisfy~\eqref{e.cost1}, then
choosing $\delta$ to satisfy the first inequality of~\eqref{e.delta-ub}
and finally choosing $\swell$ to satisfy~\eqref{e.swell-ub}.

 \begin{remark}
We made no attempt to optimize the parameters.
Not fixing them, only constraining them with inequalities has
only the purpose to leave the ``machinery'' somewhat open to later
adjustments.
Considering $\Q$ as a variable,
these relations imply $\swell \sim \Q^{-2}$.
With more careful analysis one could certainly achieve $\swell\sim \Q^{-1}$.
 \end{remark}

 \begin{lemma}[Main]
If our parameters satisfy the above relations then there is an
implementation $(\phi,J)$.
 \end{lemma}

We prove this lemma in the following sections, and then apply it to
prove the theorem.

\section{Relations among parameters}

The following simple lemma illustrates the use of cleanness.

 \begin{lemma}
If a run is clean then between any two safe points of it there is a walk.
 \end{lemma}
 \begin{proof}
A clean run contains at most one unsafe colony, the obstacle.
Suppose it is $\U_{\k}$ with $1<\k<\n$.
Then we can make $n_{i}=i$ for all $i<\k$, then $n_{i}=i+1$ for 
$\k\le i < \n$.
 \end{proof}

 \begin{lemma}\label{l.nu-clean}
The following relations hold for our thresholds.
 \begin{enumerate}[(a)]

  \item\label{i.cleani}
The time passed during any move is at most $\growth\B$.

  \item\label{i.success} If an attack
has a $1$-good reduced body then it succeeds.

 \end{enumerate}
 \end{lemma}
 \begin{proof}
Let us prove~\eqref{i.cleani}.
Suppose that the move takes time $x$ and has body $\M$.
Let $\mu_{0}$ be the measure before the move and $\mu_{1}$ after it.
By the requirements,  $\mu_{0}(\M) < 3\B$.
Note that $\tau_{\gc}(\chi)<6\B$ for the history $\chi$
consisting of the single move in question.
Now therefore we have
 \begin{align*}
     x &\le \cost_{1}\mu_{1}(\M) + \tau_{\gc}(\chi)
       \le \cost_{1}(\mu_{0}(\M) + \swell x) + 6\B
       \le \cost_{1}(3\B + \swell x) + 6\B,
\\   x &\le \frac{\B(6 + 3\cost_{1})}{1-\cost_{1}\swell}
                       \le 2\B(6+3\cost_{1}) = \B\growth,
 \end{align*}
using~\eqref{e.swell-ub} and~\eqref{e.growth}. 

To prove~\eqref{i.success} note that a $1$-good run is still good after
the move due to $\delta > \growth\swell$, and the attack could fail only
if this run became bad.
 \end{proof}

A good enough run is clean, as the lemma below shows.

 \begin{lemma}\label{l.good-clean}
If a run is $(-\Q)$-good then it is $1$-unimodal.
Consequently, if it is $(i+1)$-good then it is $i$-clean for 
$0\le i\le 1$. 
 \end{lemma}
 \begin{proof}
Let $a$ be the weight of the obstacle and $b,c$ two weights
of non-obstacle colonies.
Then we have
 \begin{align*}
  b+c  \le (2/3)(a+b+c) \le (2/3)(1+\Q\delta)\B < (\safe-\delta)\B
 \end{align*}
by~\eqref{e.delta-ub}.
The ``consequently'' part follows immediately from the definition of
$i$-cleanness.
 \end{proof}

 \begin{lemma}\label{l.good}
Suppose that for rectangle $\U$ consisting of $\Q$ horizontal runs
of colonies below each other, we have 
$\mu(\U) \le \Q\B(\safe+\delta)+\B$.
Then at least $\kappa$ of the horizontal runs are $1$-clean.

In particular, if a colony of $\bbG^{*}$ is $(-1)$-safe then at least
$\kappa$ of its rows are $1$-clean.
 \end{lemma}
 \begin{proof}
Suppose that $\U$ does not have $\kappa$ rows that are $2$-good (and thus
$1$-clean).
Then
 \begin{align*}
   (\Q-(\kappa-1))\B(1-2\delta) &< \mu(\U) < \Q\B(\safe+\delta)+\B,
\\  1-\safe -\kappa/\Q            &< \delta(3+2(\kappa-1)/\Q),
\\                    (1-\safe)/2 &< \delta(4-\safe)
&\txt{by~\eqref{e.Q-lb}},
 \end{align*}
contradicting~\eqref{e.delta-ub}.
 \end{proof}

\section{The implementation map}

This section proves the main lemma.

Let $\mu_{0}$ be the measure at the beginning of the big move,
and $\p_{0}$ the initial point.
Unless saying otherwise, the properties of parts of $\M^{*}$ are
understood with respect to $\mu_{0}$.
We will make most decisions based on the measure $\mu_{0}$.
Our map will implement each big move using at most $\nu$ small moves, where
$\nu$ was defined in~\eqref{e.nu-kappa}.
Due to Lemma~\ref{l.nu-clean} and
relation~\eqref{e.swell-ub} this will imply that, for example,
a run required to be safe with respect to $\mu_{0}$
will still be $(-1)$-safe at the end.
Let us make this statement explicit:
 
 \begin{lemma}
In any sequence of $\le\nu$ moves, the total mass increases by less than
$\delta\B$.
 \end{lemma}

This extra tolerance in the initial requirements insures
that any planned steps, jumps and turns remain executable by the time we
actually arrive at the point of executing them.
With attacks this is not the case, they can fail or we may find
immediately before executing an attack that it is
not executable anymore since the reduced body is not good anymore.

 \begin{definition}
In an implementation, 
colonies of $\bbG^{*}$ will be called \df{large} colonies, and colonies of
$\bbG$ \df{small} colonies, or \df{cells}.
Moves in the game $\bbG^{*}$ will be called \df{big} moves, and moves in
the game $\bbG$ \df{small} moves, or simply \df{moves}.
A big move has starting colony and body $\S^{*}$, $\M^{*}$.
When it has a destination colony that will be denoted by $\D^{*}$.
 \end{definition}

\subsection{Plan for estimating the time}

Let
 \begin{align*}
 \chi=(\dmove_{1},\amove_{1},\dots,\dmove_{\m+1})
 \end{align*}
be a history of $\bbG^{*}$.
As shown in Subsection~\ref{ss.scale-up},
in the implementation there corresponds to $\chi$ a history
$\chi_{*}$ of game $\bbG$, which shares the initial and final
essential configurations of $\chi$.
Segment $(\dmove_{\r},\amove_{\r},\dmove_{\r+1})$ of $\chi$
corresponds to segment
 \begin{align*}
   (\dmove_{*\s_{\r}},\amove_{*\s_{\r}},\dmove_{*(\s_{\r}+1)},\dots,
   \amove_{*\s_{\r+1}-1},\dmove_{*\s_{\r+1}})
 \end{align*}
of $\chi_{*}$.
If $\M^{*}_{\r}$ is the body of located move $\amove_{\r}$ and
$\M_{i}$ is the body of located move $\amove_{*i}$
then our implementation will give
$\M_{i}\subset\M^{*}_{\r}$ for all $\s_{\r}\le i \le \s_{\r+1}-1$.
So the body of the path of the implementation of each move is in the body
of the implemented move.

If the path $\amove(\chi)$ of history $\chi$ is simple then
we will implement it via a simple path $\amove(\chi_{*})$.
(This goal accounts for some of the complexity of the implementation.)

Since $\bbG$ is an AD-game, we can estimate the time of path $\chi_{*}$ by
the time bound introduced in Definition~\ref{d.time-bound}.
We will then show that this estimate obeys the time bound required by game
$\bbG^{*}$.
Let $\mu$ be the measure in the last configuration, let
$\U$ and $\U_{*}$ be the union of the bodies of all located moves in the
path $\amove(\chi)$ and $\amove(\chi_{*})$ respectively.
Let $\n$ be the number of of failed continuing attacks in $\chi$ and
$\n_{*}$ be the number of of failed attacks in $\chi_{*}$.
Then by the time bound of game $\bbG$ we have
 \begin{align*}
 \tau(\chi_{*}) = \cost_{1}\mu(\U_{*}) - \cost_{2}\n_{*}\B +
\tau_{\gc}(\chi_{*}).
 \end{align*}
Our goal is to show that this expression is bounded above by
 \begin{align*}
   \tau(\chi) = \cost_{1}\mu(\U) - \cost_{2}\n\Q\B + \tau_{\gc}(\chi).
 \end{align*}
Ignoring the negative terms first, as we noted $\U_{*}\subset\U$,
so of course we have $\mu(\U_{*})\le \mu(\U)$.
But $\tau_{\gc}(\chi_{*})$ will typically be larger than
$\tau_{\gc}(\chi)$.

If $\chi=\chi_{1}+\dots+\chi_{\m}$ where $\chi_{\m}$ are unit histories
and $\chi_{*i}$ is the segment of $\chi_{*}$ corresponding to $\chi_{i}$
then $\tau_{\gc}(\chi)=\sum_{i}\tau_{\gc}(\chi_{i})$,
and $\tau_{\gc}(\chi_{*})=\sum_{i}\tau_{\gc}(\chi_{*i})$.
Trying to bound each $\tau_{\gc}(\chi_{*i})$ by the
geometric cost $\tau_{\gc}(\chi_{i})$ of the big move from which it was
``translated'', we will frequently have
$\tau_{\gc}(\chi_{*i})>\tau_{\gc}(\chi_{i})$.
Let us call the difference
$\tau_{\gc}(\chi_{*i})-\tau_{\gc}(\chi_{i})$ the \df{extra geometric cost}.
The basic strategy in the implementation is to ``charge'' the extra
geometric cost to 
the weight of some sets in the difference $\U\setminus\U_{*}$.
This suffices for the translation of a big step.

Unfortunately in the implementation 
of the other moves $\chi_{i}$, there may not be enough mass outside
$\U_{*}$.
We will compensate the geometric cost 
by the negative contribution $\cost_{2}\B$ of some failed continuing
attacks in the implementation.
Of course if $\chi_{i}$ is a failed continuing attack itself then this
cannot be done, fortunately then it need not be.

 \begin{definition}
The value $\cost_{2}\B$ will be called the \df{profit} of any failed
continuing attack.
 \end{definition}

\subsection{General properties}

 \begin{definition}
Let $\U=\U_{1}\cup\dots\cup\U_{\n}$ be a vertical run of colonies.
We will say that $\U_{i}$ is \df{secure} in $\U$ provided
$\U_{i-1}\cup\U_{i}$ is safe (if $i>1$) and
$\U_{i}\cup\U_{i+1}$ is safe (if $i<\n$).

We will say that a horizontal and vertical run intersect 
\df{securely} if the intersection colony is secure either in
the horizontal run or in the vertical run.
 \end{definition}

\begin{lemma}
Let $R$ be a clean row in a rectangle.
Then there are at most 3 clean columns that do not intersect $R$ securely.
\end{lemma}
 \begin{proof}
Indeed any clean column that intersects $R$ in a position different from
the obstacle and its neighbors intersects $R$ securely.
 \end{proof}

\begin{remark}
In the procedure below, when
a row and a column intersect securely we will sometimes say that we 
first walk in the row and then continue walking in the column.
But it is understood that if the move before the
intersection is a step and the one after the intersection is a 
jump then these two moves are actually replaced by a single turn move.
\end{remark}

The destination colony of a big move is $(-1)$-safe and therefore
due to Lemma~\ref{l.good} has at least $\kappa$ rows that are $1$-clean.

 \begin{definition}\label{d.R''(i)}
Let $C(0)$ be the starting column of the angel.

When starting from a northward-clear point in $\S^{*}$,
we will denote by $R'(0)$ the row to which it is possible to step
north.
For $i\ge 1$ let $R'(i)$ denote the $(i+1)$th clean 
row of $\S^{*}$ starting from the south.

Let $R''(i)$ denote the $i$th clean row of $\D^{*}$ starting from the south
with the additional property $R''(i)>1$.
 \end{definition}

In the implementation of a big step, jump or new attack, ideally we would
just walk in column $C(0)$ to a cell below row $R''(1)$ from which
it is reachable in one step.
We will do something else only if this is not possible.

\begin{definition}
Certain runs of cells in the body of each implemented big move will be
called \df{scapegoat runs}.
Consider a history $\chi$ of $\bbG^{*}$ and its implementation
$\chi_{*}$ in $\bbG$.
We will make sure that
all scapegoat runs will be disjoint of each other as well as of the
body of $\chi_{*}$.

Certain subhistories will be called \df{digressions}.
Each digression will be \df{charged} to some scapegoat run,
and different digressions will be charged to different scapegoat runs.
\end{definition}

 \begin{definition}
We will say that row $R$ 
is \df{securely reachable} from a cell $\U$ below it
if the upward vertical 
run from $\U$ to the last cell $\U'$ 
below $R$ is clean and the step from $\U'$ to $R$ is safe.

Let us call the \df{blameable run of} $\U,R$ 
the vertical run starting from the cell above $\U$ and ending in $R$.
 \end{definition}

Let us lowerbound the weight of a blameable run.

 \begin{lemma}\label{l.blame}
Suppose that $R$ is not securely reachable from $\U$.
Let $\V$ be the blameable run of $\U,R$, and $\mu$ the current measure.
Then we have
 \begin{align*}
 \mu(\V) \ge 0.5 \B(1-\safe).
 \end{align*}
 \end{lemma}
 \begin{proof}
If the step with body $\U'\cup\U''$ 
from cell $\U'$ below $R$ to cell $\U''$ in $R$
is not safe then $\mu(\U'\cup\U'')\ge \safe\B$. 
Suppose it is safe.
Since the run $\V'$ from $\U$ to $\U'$ is not clean,
it follows from Lemma~\ref{l.good-clean} that it is not $1$-good.
Using the fact that the angel's current position
is in a $(-1)$-safe cell,
 \begin{align}\label{e.blame}
         \mu(\V') &\ge \B(1-\delta),
\\        \mu(\U) &\le \B(\safe+\delta).
\\ \mu(\V'\setminus\U) &\ge \B(1-\safe-2\delta) \ge 0.5\B(1-\safe)
 \end{align}
due to~\eqref{e.delta-ub}.
 \end{proof}

So a blameable run has weight $\ge 0.5 \B(1-\safe)$.
If there is extra geometric cost (at most $\c\Q\B$ for some constant $c$),
then the inequalities in Subsection~\ref{ss.main-lemma} show that
we will be able to charge it all against $\rho_{1}$ times
this weight, the weight of an appropriate scapegoat run.

(This is crude, a factor of $\Q$ is lost here unnecessarily.)

\subsection{Big step}

Let us define now the translation of a big northward step.

\begin{enumerate}[1.]

 \item\label{i.step.in-clean-row}
Let us call the starting cell $\U$.
If some $R''(i)$ is securely reachable from $R'(0)\cap C(0)$
then let $R''(i''_{0})$ be the first such (note that $R''(i''_{0})>1$).
We will walk to the colony below $R''(i''_{0})$ in $C(0)$.
Row $R''(i''_{0})$ will still be clean, so we will be done.

Suppose now that $R''(1)$ is not securely
reachable from $R'(0)\cap C(0)$.
Since the body of the big step is $(-1)$-safe, 
Lemma~\ref{l.good} implies that it has at least
$\kappa$ columns that are clean.

Let $C(1)$ be one that is securely reachable from $R'(0)\cap C(0)$.
Let  $i''_{1}$ be the first $i>1$ 
such that $R''(i)$ is securely reachable from $R'(0)\cap C(1)$.
We step up to $R'(0)$, then walk to
$R'(0)\cap C(1)$ and then to the cell below $R''(i''_{1})\cap C(1)$.
The scapegoat is the blameable run of $R'(0)\cap C(0)$, $R''(1)$.

 \begin{remark}
We choose $i>1$ in order to be above $R''(1)$, since a cell of $R''(1)$ may
be part of a scapegoat run, and the clean row under which we will end up
should be disjoint from the scapegoat run, in order to be avoid
intersecting the scapegoat run in the implementation of the next big move.
In what follows we choose larger and larger $i''$ values for similar reasons.
 \end{remark}

The total geometric cost is at most $\B(16 + \Q + 2\Q)\le 4\Q\B$
(with two turns and two straight runs), 
and the total number of moves is at most $3\Q$.

 \item\label{i.step.in-clean-column} 
Suppose that the starting point is not northward clear,
but it is eastward clean, let $C(1)$ be the clean
column one step to the right of the current cell $\U$.

Again, if some $R''(i)$ is securely reachable from $\U$
then let $R''(i''_{0})$ be the first such.
We will walk up to the colony below $R''(i''_{0})$.

 \begin {figure}
\begin {equation*}
\includegraphics{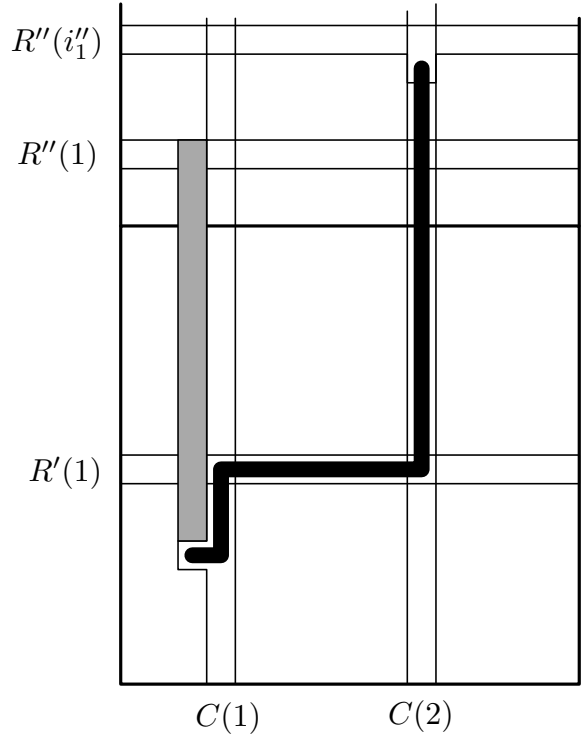}
\end {equation*}
\caption{Implementing a northward step from an eastward-clear point.
The grey rectangle is the scapegoat.
 \label{f.step}}
 \end {figure}

If we do not have this case (see Figure~\ref{f.step}) then
since the body of the big step is $(-1)$-safe, again there is
a clean column $C(2)$ of $\M^{*}$, to the east of $C(1)$.
Let $i''_{1}$ be the first $i>1$ such that $R''(i)$ securely
intersects $C(2)$.
There is a clean row $R'(1)$ of $\S^{*}$
securely intersecting both $C(1)$ and $C(0)$.
We will step into $C(1)$, then walk to $R'(1)\cap C(1)$,
then to $R'(1)\cap C(2)$, finally walk 
in $C(2)$ to the cell below $R''(i''_{1})$.
The scapegoat is the blameable run of $\U,R''(1)$.

The total geometric cost is at most $\B(24 + \Q + \Q + 2\Q)\le 5\Q\B$, 
adding up the cost
of at most three turns and three runs, and the total number of moves is at
most $4\Q$.
 \end{enumerate}

 \subsection{Big jump}

Consider a big northward jump, starting from a northward clear point.
There is a clean row $R'(0)$ to which it is possible to step north
from the current cell $\U$.
We will use Definition~\ref{d.R''(i)} again.

\subsubsection{If a clean column exists}

Suppose that $\M^{*}$ has a clean column $C(1)$.
Then we will proceed somewhat as in part~\ref{i.step.in-clean-row}
of the implementation of a big step, with possibly yet another 
digression.

If some $R''(i)$ is securely reachable from $R'(0)\cap C(0)$
then let $R''(i''_{0})$ be the first such.
We will walk to the colony below $R''(i''_{0})$ in $C(0)$.
Row $R''(i''_{0})$ will still be clean, so we will be done.

Suppose now that $R''(1)$ is not securely
reachable from $R'(0)\cap C(0)$.
If $C(1)$ was securely reachable from $R'(1)\cap C(0)$ then we could proceed
just as in the implementation of a big step, but this is not guaranteed
now, so suppose it is not so.
Let $R'(i'_{1})$ be the lowest clean row of $\S^{*}$ above
$R'(0)$ securely intersecting $C(1)$ and let $C(2)$ be a column 
whose intersection colony is secure both with $R'(0)$ and $R'(i'_{1})$.
We step up to $R'(0)$, then walk in $R'(0)$ to $C(2)$, further
in $C(2)$ to $R'(i_{1})$.

If some $R''(i)$ for $i>1$ is securely reachable from $R'(i'_{1})\cap C(2)$
then let $R''(i''_{2})$ be the first such.
We will walk to the colony below $R''(i''_{2})$ in $C(2)$,
charging everything to the scapegoat run of $\U,\R''(1)$.
Otherwise let  $i''_{2}$ be the first $i>1$ 
such that $R''(i)$ is securely reachable from $R'(i'_{1})\cap C(1)$.
We walk in $R'(i'_{1})$ to $C(1)$,
and then $C(1)$ to the cell below $R''(i''_{2})$.
The whole operation will be charged to the scapegoat run of
$R''(i''_{1})\cap C(2), R''(i'_{1})$, 

The geometric cost in this worst case is at most $\B(32+3\Q+3\Q)\le 7\Q\B$,
the number of moves is at most $6\Q$.
We will be able to charge the geometric cost to a single scapegoat
run due to~\eqref{e.cost1}.

From now on we suppose that no clean column exists in $\M^{*}$.

\subsubsection{Obstacles}\label{sss.obstacles}

Let us draw some consequences of the fact that the body $\M^{*}$ is good.
(For a big jump we actually have 1-goodness, but the analysis will also be
applied to the implementation of attacks.)

 \begin{lemma}\label{l.three-cases}
If $\M^{*}$ has no $1$-good column then every column is unimodal.
 \end{lemma}
 \begin{proof}
Suppose that no column is $1$-good.
Let $w_{1}\le w_{2}\le\dots\le w_{\Q}$ be the weights of all the columns,
ordered, so $w_{i}\ge (1-\delta)\B$ for all $i$.
Then, 
 \begin{align*}
   \Q\B \ge w_{1}+\dots+w_{\Q} &\ge (\Q-1)\B(1 - \delta) + w_{\Q},
\\        \B(1 + (\Q-1)\delta) &\ge w_{\Q}.
 \end{align*}
Now unimodality is implied by Lemma~\ref{l.good-clean}.
 \end{proof}

So now we know that each column is unimodal.
On the other hand, 
if we still manage to pass through in a simple way then there
 plenty of ways to charge it, since all columns of $\M^{*}$ are heavy.

Suppose that for some $i\ge 1$ there is a
column $C(1)$ of $\M^{*}$ whose obstacle is below $R'(i)$.
It is easy to perform an implementation that gets us to $R'(i)$ via some
column $C(2)$ that is clean in $\S^{*}$ and intersects $R'(0)$ and $R'(i)$
securely, then from there to $C(1)$, and finally walks north on $C(1)$ to an
appropriate row $R''(i'')$.
Charging is done like in earlier similar cases
(but as mentioned above is not a problem anyway).
The extra geometric cost and the number of moves have also the same bounds.
The case remains that no obstacles are below any $R'(i)$.

Suppose that there is some obstacle above $R''(1)$, in some column $C(1)$.
Then we can walk to $C(1)$ and then in $C(1)$ to below 
$R''(1)$, using again possibly an intermediate
row $R'(i')$ and column $C(2)$.
Charging and bounding the extra
geometric cost and number of moves is done as before.
The case remains that there are no obstacles below 
$R'(i)$ for any $i$ and above $R''(1)$.

Let us make another observation about obstacles.
 \begin {figure}
\begin {equation*}
\includegraphics{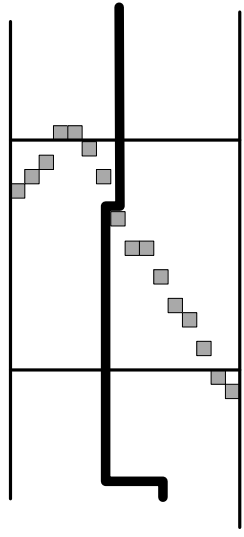}
\end {equation*}
\caption {The small grey squares are obstacles.
The path passes between two neighbor obstacles that are not close in height.
\label{f.obstacles}}
 \end {figure}

  \begin{lemma}\label{l.distant-obstacles}
Suppose that there are $i,j$ such that $r_{j},r_{j+1}\not\in\{i,i+1\}$.
Then there is a $q\in\{0,1\}$ such that
the horizontal run $\M_{i+q,j}\cup\M_{i+q,j+1}$ is safe.
In other words, it is possible to pass horizontally between the two
obstacles.
  \end{lemma}
  \begin{proof}
 Since the runs $\M_{i+1,j}\cup\M_{i+2,j}$ and
 $\M_{i+1,j+1}\cup\M_{i+2,j+1}$ are safe, the measure of the union of these
 four cells is $<2\B\safe$, hence one of the horizontal runs
 $\M_{i+q,j}\cup\M_{i+q,j+1}$ with $q\in\{1,2\}$
 has measure $<\B\safe$.
  \end{proof}

Consider the case when there is a $j$ with $|r_{j}-r_{j+1}|>2$.
We can escape then as follows.
Assume without loss of generality $r_{j+1} < r_{j} - 2$.
Then by the lemma there is a $q\in\{1,2\}$ such that
the run $\M_{r_{j}-q,j}\cup\M_{r_{j}-q,j+1}$ is safe.
Therefore one can walk to column $j$
(again possibly using an intermediate column $C(2)$ and row
$R'(i')$), then
up to $(r_{j}-q,j)$, pass to $(r_{j}-q,j+1)$ and further up in column
$(j+1)$.
The geometric cost is at most $\B(40 + 2\Q + 3\Q)\le 6\Q\B$, 
the number of moves is at most $5\Q$.
Charging is done as usual, but as remarked above is not a problem anyway.

 \begin{definition}\label{d.straight}
We will say that in a big jump, the body $\M^{*}$, or in a big attack, 
the reduced body $\ul\M^{*}$ has the \df{marginal case} if the following
holds: 
 \begin{enumerate}[(a)]
  \item Every column is unimodal.
  \item The rows $R'(i)$ contain no obstacle.
  \item For all $j$ we have $|r_{j}-r_{j+1}|\le 1$.
 \end{enumerate}
Otherwise we have the \df{straight case}.
 \end{definition}

In summary, we found a strategy for the straight case with geometric cost
at most $\B(40 + 5\Q)\le 6\Q\B$ and at most $5\Q$ moves.

Assume now that we have the marginal case.
In this case, attacks will be performed.
We need some cells to blame in the case of an attack that became
disallowed.

 \begin{definition}
Let $r_{j}$ denote the height of the obstacle in column $j$ of $\M^{*}$.
A position $(i,j)$ is called \df{northwards bad} if the run in column $j$
starting with it and ending in row $R''(1)$ contains the obstacle of a
disallowed attack, which is also the obstacle cell of a column.
This obstacle cell will be called a \df{scapegoat cell}.
 \end{definition}

The lemma below lowerbounds the weight of the scapegoat cell.

 \begin{lemma}\label{l.no-attack}
The weight of the scapegoat cell is lowerbounded by $(1-\delta)\B/6$.
 \end{lemma}
 \begin{proof}
Let $\mu_{\t}\ge \mu_{0}$ 
be the measure at the time when the attack is disallowed.
Let $\U=\U_{1}\cup\dots\cup\U_{6}$ be the reduced body of 
the attack, and let $\U_{i}$ be the scapegoat cell.
Then $\mu_{0}(\U_{i})\ge (1-\delta)\B/6$.
Since each small move has weight $\le 3\B$ and since there will be
at most $\nu$ small moves per big move, and due to~\eqref{e.swell-ub},
during the implementation the weight of $\U$ could increase by at most
$3\nu\swell\growth\B\le\delta\B$, hence we have
 \begin{align*}
 \B\le \mu_{t}(\U) \le \mu_{0}(\U) +\delta\B \le 6\mu_{0}(\U_{i}) + \delta\B.
 \end{align*}
 \end{proof}

\subsubsection{Preparing a sweep}

In the marginal case, we have no obstacles in the clean row $R'(0)$ 
above the starting cell.
We step up to this row.
If row $R''(1)$ is securely reachable from some cell of $R'(0)$ then we can
pass there and charge our costs again as usual.
Suppose that this is not the case.

We pass to column 1.
Now we are below the obstacle in column 1.
We step up to height $i_{0}=\min(r_{1}-2,r_{2}-2)$, then 
$r_{1}-4\le i_{0}$.
The geometric cost of these preparatory steps is at most 
$\B(8 + \Q + 3\Q)$, and we make at most $4\Q$ moves.

Now in a series of northward attacks passing to the right we will
explore the obstacles.
We will call this latter series a \df{right sweep}.
At the latest, in column $\Q$ our attack will succeed.
Indeed, each time we pass to the right the column we have left is bad,
and the conditions of the big step (that the body is $(-1)$-good)
do not allow that the whole big body
become bad by the end of the implementation.

A sweep will be an alternation of procedures called
\df{success branch} and \df{failure branch}.

 \begin {figure}
\begin {equation*}
\includegraphics{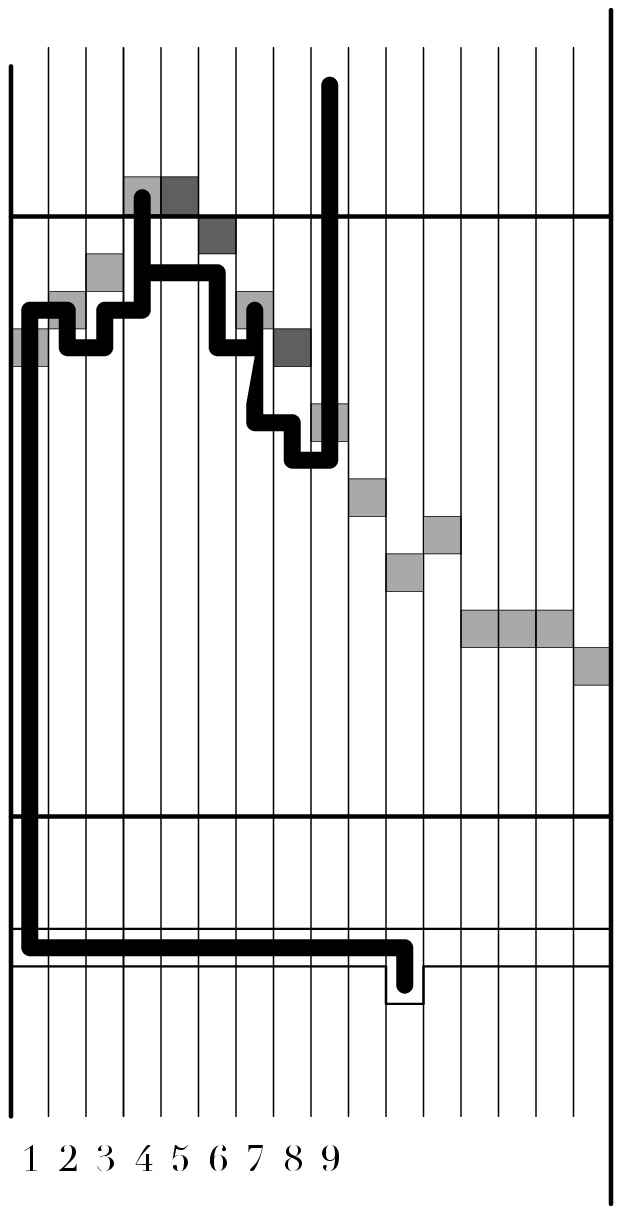}
\end {equation*}
\caption {Implementing a sweep.
Here is a description of what happens in the first few columns.
1: Failed new attack, continuing attack.
2,3: Failed continuing attack, followed by new continuing attack.
4: Failed continuing attack followed by finish, right step.
5,6: evasion.  The darker grey squares are scapegoat cells.
7: Failed new attack followed by finish, right step.
8: evasion. 9: Successful new attack.
\label{f.attack}}
 \end {figure}

\subsubsection{Success branch} 

\begin{remark}
When the procedure below calls for an eastward step followed by a northward
new attack, it is understood that actually the two are combined into a turn
move. 
\end{remark}

A success branch starts after a successful move.
If it was a successful attack we escape, at no extra geometric cost.
Otherwise we are either before the initial attack in column 1, or are
coming from the left.

We are in some column $j$, in some current row $i_{0}$, with
$r_{j}-5 \le i_{0}<r_{j}$.
If a new attack is allowed (namely the body of that attack is 
good), then we make it, of level $i_{0}-r_{j}$.
We end up at a height $i_{1}$ with $|i_{1}-r_{j+1}|\le 1$.
If it is not allowed then if $j=\Q$ we halt, otherwise what follows will be
called an \df{evasion}.

By assumption $r_{j}-2 \le r_{j+1} \le r_{j}+2$.
If $r_{j+1} < r_{j}$
then we step up or down from $i_{0}$ to $r_{j+1}-q$ for some
$q\in\{1,2\}$ and pass to the right.
If $r_{j+1} \ge r_{j}$ then we step up or down from $i_{0}$ to
$r_{j}-q$ for some $q\in\{1,2\}$ and pass to the right.
(Note that we never have to pass down if we are in column 1, since there we
are at height $\min(r_{1}-2,r_{2}-2)$.)
In both cases we end up at a new height $i_{1}$ with
$r_{j+1}-4\le i_{1}<r_{j+1}$.

An extra $\cost_{2}\B$ will also be charged to that scapegoat cell,
to compensate for the profit that we do not have since we did not have a
continuing attack.
Inequalities~\eqref{e.cost2} and~\eqref{e.cost1} show 
that the scapegoat cell has enough weight for these charges.
(This part is not important for the case of a big step, but we will reuse
the analysis of the marginal case of the big step in the case of a
failed attack, where the profit is needed.)

A success branch spends at most $5$ moves in each column (5 in case of
evasion, 1 in case of new attack).
If we made a new attack but we came from an evasion,
we charge the geometric cost of the new attack and the missing profit
to the scapegoat cell above the northwards bad position in the previous
evasion.
Inequalities~\eqref{e.cost2} and~\eqref{e.cost1} show 
that the cell has enough weight
for these charges even in addition to the charges made for the evasion
itself.

\subsubsection{Failure branch}

A failure branch starts after a failed attack, so $|i_{0}-r_{j}|\le 1$.
Let $i_{1}$ be height of the bottom cell of the body of that attack.
We halt if $j=\Q$ (never happens in the implementation of a big step).
If $i_{0} - r_{j+1} > 1$ then we make an escape move and then escape,
at no extra geometric cost.
Suppose $i_{0}-r_{j+1}\le 1$.
Since $r_{j+1}\le r_{j}+2$ we have $-3 \le i_{0}-r_{j+1} \le 1$.
If a continuing attack with obstacle $r_{j+1}$
is allowed we make it, at no extra geometric cost or lost profit.
Suppose it is not allowed.
Let
 \begin{align*}
  i_{2} =  \min(i_{0}-1,r_{j}-1,r_{j+1}-1).
 \end{align*}
By Lemma~\ref{l.distant-obstacles}
a step right is possible at height $i_{2}-q$ for some
$q\in\{0,1\}$.
We can get to $i_{2}-q$ using a finish move taking us to
$\max(i_{1},i_{2}-q)$ and following it with some downward steps.
After moving to column $(j+1)$
we are positioned for a success branch, at a position $(i_{3},j+1)$
with $r_{j+1}-4 \le i_{3} < r_{j+1}$.

A failure branch spends at most $6$ moves per column, at a geometric
cost of at most $7\B$ which we will charge to the scapegoat cell of the
disallowed attack.
We will even charge the next evasion to it, which will still be allowed 
by~\eqref{e.cost1}.

The final escape has no geometric cost and takes at most $3\Q$ moves.

\subsubsection{Summary of costs of the marginal case of a big step}

In the marginal case the extra geometric cost is
at most $\B(5 + 4\Q)\le 5\Q\B$ for the part before the sweep
(the one during the sweep is accounted for).
The number of moves is at most $4\Q + 7\Q + 3\Q=14\Q$, where we also
counted the moves of the final escape.

If we succeed before column $C_{\Q}$ then we can charge the extra geometric
cost to the blameable run from $R'(0)\cap C_{\Q}$ to $R''(1)$.
Otherwise we turn a profit in all columns but possibly
the first one (not a continuing attack)
and the last one (not a failed attack), 
so we can charge the extra geometric cost to
this profit of size $(\Q-2)\cost_{2}\B$, via~\eqref{e.cost2}.

\subsection{Big attack}

\subsubsection{New attack}

We only look at northward attacks passing to the right.
A new northward attack
is implemented just like a jump: the only difference is
that it may actually fail.
We say that the attack fails if either the last attack in column $\Q$ fails
or we find ourselves in an northwards bad position in the last column.
\emph{This is the definition of the function $\J(\cdot)$ of the
  implementation}. 

Also, row $R'(i)$ is defined, instead of Definition~\ref{d.R''(i)},
as the $(i+1)$th southernmost $1$-clean
row of the colony immediately below the obstacle colony of the big attack.
Then we can conclude that in the marginal case
all obstacles are above even these rows $R'(i)$.

The bounds on the geometric cost change only by the consideration
the body of an attack may be by $3\Q$ longer than that of a jump.
So the extra geometric cost will be bounded by
$\B(8 + 7 \Q)\le 8\Q\B$ and the number of moves by $17\Q$.

Suppose therefore that 
our attack with body $\M_{1}^{*}$ is a continuing one, then the previous
move was a failed attack, also to the north and passing to the right,
with some body $\M_{0}^{*}$.
The last move
of its implementation was either a failed attack, or a successful move
ending in an northwards bad position.
We are in the rightmost column of $\ul\M_{0}^{*}$.

\subsubsection{Marginal case}

If the attack of $\ul\M_{1}^{*}$ has
the marginal case, as defined in Definition~\ref{d.straight}
then its columns are unimodal.
In this case we can just continue the sweep to the right into
$\ul\M_{1}^{*}$ seamlessly, except that in the first column
of $\ul \M_{1}^{*}$ we may have to walk up as in case of the first column
of the marginal case of a big jump.
(One could also just step back to the last column of $\ul\M_{0}^{*}$ and
escape, but we will pass this opportunity, for the sake of orderliness.)

The walk-up entails no extra geometric cost.
The cost of the rest of the implementation will be estimated
just as for new attacks.
If the attack is a failed one we always have the marginal case, and
we do not escape.
In this case each column contributes its required profit: the ones that did
not 
were compensated by others, as shown in the implementation of the marginal
case of a big step.

\subsubsection{Straight case}

Assume now that the the attack of $\ul\M_{1}^{*}$ has
the straight case, and we are at position $(\m,\Q)$ of $\ul\M_{0}^{*}$.

 \begin{enumerate}[1.]
  \item
Suppose that the northwards 
run of column 1 of $\ul\M_{1}^{*}$ starting from position $(\m-1)$
to row $R''(1)$ is safe.
If the last step of the big attack of $\ul\M_{0}^{*}$ 
was a failed attack we make an escape move into $\m$ and then escape,
at no extra geometric cost.

If the last step was a successful move then
the starting position is northwards bad.
We step right if
this is possible, otherwise we step down and step right: just as in the evasion
procedure, this is always possible.
Then we escape north.
The geometric cost of these digression steps is charged to the scapegoat
cell above the northwards bad starting position.
(Note that the body of the big escape and of each big continuing attack
contains a big colony above the starting big colony, so the scapegoat cell
is inside this body.)

  \item
Suppose now that the northwards 
run of column 1 of $\ul\M_{1}^{*}$ starting from position $(\m-1)$
is not safe.
In case the last step was a failed attack we add a southward finish move
taking us to the bottom of the reduced body of that attack.
Now we are in a southward-clear point in a step-clean run below the
obstacle of the last column of $\ul\M_{0}^{*}$, which stretches 
down at least two big colonies below the current big colony.
One of these big colonies, say $\U_{0}^{*}$, is such that
its right neighbor $\U_{1}^{*}$ 
is below the obstacle of the big attack of $\ul\M_{1}^{*}$.

Let $\V^{*}$ be the set consisting of $\U_{1}^{*}$ and 
the last column of $\U_{0}^{*}$.
This set has at least $\kappa$ clean rows.
Indeed, since $\U_{1}^{*}$ is safe, we have $\mu(\U_{1}^{*})\le\safe\B\Q$.
For the last column $C_{\Q}$ of $\U_{0}^{*}$ we have
$\mu(C_{\Q})\le (1+\Q\delta)\B$.
So we have
 \begin{align*}
 \mu(\V^{*}) \le \B\Q\safe + \B + \B\Q\delta,
 \end{align*}
therefore Lemma~\ref{l.good} is applicable.

 \begin{enumerate}[a.]
 \item
Suppose that there is a  column $C(1)$ of $\ul\M^{*}$ different from the
first column, a clean row $R(1)$ of $\V^{*}$ that intersects both
$C_{\Q}$ and $C(1)$ securely, and an $i''_{1}>1$ such that
$R''(i''_{1})$ is securely reachable in $C(1)$ from $R(1)$.
(This is always the case if $\ul\M_{1}^{*}$ has a $1$-good column
different from the first one.)
We then walk down into $R(1)$, then walk over to $C(1)$, 
and finally walk up in $C(1)$ to under $R''(i''_{1})$.
The geometric cost is at most $\B(16+2\Q+\Q+6\Q)=\B(16+9\Q)\le 10\Q\B$, and the
number of moves is at most $9\Q$.
We charge all this to the unsafe run from $(\m-1,1)$ to $R''(1)$.

  \item 
Suppose now that the above case does not hold.
Let $R(1)$ be any clean row of $\ul\M_{1}^{*}$ securely intersecting
$C_{\Q}$, and
let $C(1)$ be any column of $\ul\M_{1}^{*}$ different from the first column,
securely intersecting $R(1)$,
then  $R''(2)$ is not securely reachable from $R(1)\cap C(1)$.

If column 1 of $\ul\M_{1}^{*}$ is $1$-good
then there is an $i''_{1}>2$ such that $R''(i''_{1})$ intersects
column 1 securely.
We can move over to column 1 and escape to the cell below $R''(i''_{1})$,
charging its costs (similar to the above ones)
to the scapegoat run of $R(1)\cap C(1)$, $R''(2)$.

 \item
Suppose finally that no column of $\ul\M_{1}^{*}$ is $1$-good, but that
there is a $j$ such that $|r_{j+1}-r_{j}|>2$.
We can then escape similarly to how we did in the big jump.
The geometric cost is at most $\B(40 + 10\Q)\le 11\Q\B$, the 
number of moves is at most $10\Q$.
Charging is again not a problem just as there.

 \end{enumerate}

 \end{enumerate}

\subsection{Big escape or big finish}

A big escape is implemented just like a big continuing attack;
the only difference is that we always have the straight case
since the reduced body is safe.

A big finish is applied under some circumstances
after a failed big attack with reduced body $\ul\M_{0}^{*}$.
The last move of the implementation of the big 
attack was either a failed attack, or a rightward step
ending in a northwards bad position in the rightmost column
of $\ul\M_{0}^{*}$.
If it was a failed attack we apply a small southward finish to arrive
at the bottom cell of the failed attack or one cell below it,
to be in a southward-clear point.
In both cases then we step down to the cell above
the second highest $1$-clean row of the lowest colony of the big finish move.
(The end result must be southward-clear.)

There is no extra geometric cost.
The number of moves is at most $6\Q$.

\subsection{Big turn}

A northward-eastward turn consists of a northward step
with body $\M_{0}^{*}$ followed by an eastward jump or northward-sweeping
eastward attack with body $\M_{1}^{*}$. 
Of course, the destination colony $\D_{0}^{*}$ of the first part coincides
with the start colony $\S_{1}^{*}$ of the second part.

Since the second part of the turn is a big eastward jump, 
in its discussion what were columns of the discussion of a
big northward jump become rows and vice versa.
Initially we are below row $R'(0)$ of $\S_{0}^{*}$ and in column
$C(0)$ of $\M_{0}^{*}$.

The key to the implementation of the turn is that there are $\kappa$ 
clean columns of the big step, so we can direct the
implementation of the big step in such a way as to arrive into
an appropriate column of $\ul\M_{1}^{*}$.

 \begin{enumerate}[1.]

  \item Suppose that we have the marginal case of $\M_{1}^{*}$.
Then we implement the big step in such a way that we arrive into
the $\M_{1}^{*}$ along a clean column that crosses the first row of
$\M_{1}^{*}$ securely.
Then we turn east and after walking right near the obstacles
begin a northward sweep of a series of eastward
attacks as in the implementation of an ordinary big eastward jump.

  \item Suppose that we have the straight case of $\M_{1}^{*}$.
We then can direct the implementation of the big step $\M_{0}^{*}$ 
in each subcase in 
such a way that we will escape similarly to the corresponding
subcase of the straight case of a big jump.
For example if $\M_{1}^{*}$ has a $1$-good row $R(1)$
then we will arrive
along a $1$-good column of $\M_{0}^{*}$ that intersects $R(1)$
securely.
 \end{enumerate}

We charge the extra geometric cost of the step to the geometric cost
of the turn which is defined as $8\Q\B$ instead of $5\Q\B$
to accomodate this.
The extra geometric cost of the jump can be charged as before.

This concludes the construction.
It is easy to check that it satisfies the requirements of an
implementation and thus the proves of the main lemma.

\section{Nested strategies}

With the proof of the main lemma,
it will be clear to some readers that the angel has a winning
strategy.
What follows is the formal definition of this strategy based on
the implementation map.
Let us first define the notions of strategy used.

 \begin{definition}
Let $\Configs^{+}$ be the set of nonempty finite sequences of
configurations.
We will use the addition notation
 \begin{align*}
 \gamma' + \dmove = \gamma
 \end{align*}
to add a new configuration to a sequence $\gamma'$.
A \df{plain strategy} is a map
 \begin{align*}
   \eta: \Configs^{+} \to \Lmoves
 \end{align*}
giving the angel's next move after each d-history.

We will call a plain strategy \df{winning} if it has the following
property.
For $\n>1$, let 
$(\dmove_{1},\amove_{1},\dots,\dmove_{\n},\amove_{\n})$ be an a-history
in which 
 \begin{enumerate}[(a)]
  \item $\dmove_{1}$ is the default configuration, and
for each $i>1$, $\dmove_{i}$ is a permitted move of the devil after the
a-history $(\dmove_{1},\dots,\amove_{i-1})$.
  \item For each $i\le n$, we have $\amove_{i}=\eta(\dmove_{1},\dots,\dmove_{i})$.
 \end{enumerate}
Then $\amove_{\n}$ is a permitted move of the angel after the d-history
$(\dmove_{1},\amove_{1},\dots,\dmove_{\n})$.
 \end{definition}

We will define a winning plain strategy with the help of nested strategies, which
will be constructed with the help of the implementation map.
We can scale the map of the lemma 
into maps $(\phi_{\k},\J_{\k})$ and then use it to obtain an
an amplifier for $\Phi_{1},\Phi_{2},\dots$.
where $\bbG_{\k+1}=\bbG_{k}^{*}$, and $\bbG_{k}$ has colony size 
$\B_{k}=\Q^{k}$.

 \begin{definition}\label{d.stack}
A \df{stack} for game $\bbG_{\k}$ 
is a finite nonempty sequence of nonempty legal a-histories
$(\chi_{1},\dots,\chi_{\m})$, where
$\chi_{i}\in\Histories_{a}(\B_{i+\k-1})$,
and if $\m>1$ then the last history $\chi_{\m}$ is not the default
record.
It is understood that this finite sequence of a-histories
stands for the infinite sequence in which each $\chi_{i}$ with
$i>\m$ is the default a-history $(\record_{0})$.
Let
 \begin{align*}
   \Stacks_{\k}
 \end{align*}
be the set of all possible stacks for game $\bbG_{\k}$.
The interpretation of a stack can be that $\chi_{i}$ for $i\ge\k$ 
is the history of game $\bbG_{i}$ played so far, in a
translation of the last step of game $\bbG_{i+1}$.
(This interpretation imposes more restrictions on the possible stacks, but
we do not need to spell these out.)
 \end{definition}

 \begin{definition}
A \df{nested strategy} for game $\bbG_{\k}$ is a map
 \begin{align*}
        \psi : \Configs\times\Stacks_{\k} &\to \Stacks_{\k},
\\   (\dmove,(\chi_{1},\dots,\chi_{\m})) &\mapsto
   \psi(\dmove \mid \chi_{1},\dots,\chi_{\m}).
 \end{align*}
Here $\psi(\dmove \mid \chi_{1},\dots,\chi_{\m})
 = (\chi'_{1},\dots,\chi'_{\n})$ with $\n\in\{\m,\m+1\}$.
 \end{definition}

The interpretation of a nested strategy for game $\bbG_{1}$
is the following.
Consider an amplifier $\bbG_{1},\bbG_{2},\dots$.
Then the angel of game $\bbG_{1}$
uses the the current configuration, further the following
histories: 
the history $\chi_{1}$ of game $\bbG_{1}$ since the beginning of the last
step 
of game $\bbG_{2}$ (the earlier history of game $\bbG_{2}$ is not needed),
the history $\chi_{2}$ of game $\bbG_{2}$ since the beginning of the last
step of game $\bbG_{3}$, and so 
on.\footnote{The reader may be amused by a faintly 
analogous idea in the poem Ajedrez (Chess) II by Jorge Louis Borges
(findable on the internet).
Borges refers back to Omar Khayyam.}
The strategy computes the new next step of the angel
and the corresponding new stack of histories.
The following definition describes how this is done in our case.

 \begin{definition}
Assume that we are 
given an amplifier $\bbG_{1},\bbG_{2},\dots$ with implementation maps
$(\phi_{1},\J_{1})$, $(\phi_{2},\J_{2})$, $\dots$.
We define a nested 
strategy $\psi_{\k}$ for each game $\bbG_{\k}$ (the functions
$\psi_{\k}$ for different $\k$ actually differ only
in the scale $\B_{\k}$).
We want to compute 
 \begin{align*}
       \Psi = \psi_{\k}(\dmove \mid \chi_{1},\dots,\chi_{\m}).
 \end{align*}
The definition is by induction on the length $\m$ of the stack.
If $\m=1$ then let $\chi_{2}=(\record_{0})$, the default a-history.
Let $\record(\chi_{2})$ be the last record of $\chi_{2}$, and
 \begin{align}\label{e.level-1}
   \amove = \hat\phi_{\k}(\chi_{1}+\dmove \mid \record(\chi_{2})).
 \end{align}
If $\amove$ is not the halting move then
$\Psi = (\chi_{1}+\dmove+\amove,\chi_{2},\dots,\chi_{\m})$.
Assume we have the halting move.
Then we set
 \begin{align}
\nonumber
  \dmove^{*} &= \hat\J_{\k}(\dmove\mid\chi_{1},\dots,\chi_{\m}),
\\\label{e.recurse}
   (\chi'_{2},\dots,\chi'_{\n}) &=
 \psi_{\k+1}(\dmove^{*} \mid  \chi_{2},\dots,\chi_{\m}),
\\\label{e.halting-case}   
   \amove' &= \hat\phi_{\k}((\dmove) \mid \record(\chi'_{2})),
\\\nonumber   \Psi &= ((\dmove,\amove'),\chi'_{2},\dots,\chi'_{\n}).
 \end{align}
In case $\m=1$ the 
step~\eqref{e.recurse} does not lead to another recursive step.
Indeed, then $\chi_{2}=(\record_{0})$ and then
step~\eqref{e.level-1} gives 
$\hat\phi_{\k+1}(\record_{0}+\dmove^{*} \mid \record(\chi_{2}))$.
According to condition~\eqref{i.when-halt} of Definition~\ref{d.phi},
the result here cannot be the halting move.
Step~\eqref{e.halting-case} cannot yield the halting move either.
 \end{definition}

It is easy to check by induction that the output of $\psi_{\k}$ is indeed a
stack satisfying the requirements of Definition~\ref{d.stack}.

Let us derive a winning plain strategy from a nested strategy.

 \begin{definition}
Let $\psi$ be a nested strategy for game $\bbG_{1}$.
We define a plain strategy $\eta(\gamma)$ for $\bbG_{1}$.
As mentioned in the definition of plain strategies we only consider
response histories
$\gamma$ which start with the default configuration $\dmove_{0}$.
We will make use of an auxiliary function 
 \begin{align*}
 \hat\eta : \Configs^{+} \to \Stacks_{\k}.
 \end{align*}
Then if $\hat\eta(\gamma)=(\chi_{1},\dots,\chi_{\m})$ 
we define $\eta(\gamma)$ as the last move of $\chi_{1}$.

The definition of $\hat\eta$ is by induction.
For a sequence consisting of a single configuration we define
 \begin{align*}
   \hat\eta((\dmove_{0})) = \psi(\dmove_{0} \mid (\record_{0})).
 \end{align*}
If $\gamma= \gamma'+\dmove$ then let 
 \begin{align*}
  (\chi_{1},\dots,\chi_{\m}) &= \hat\eta(\gamma'),
\\          \hat\eta(\gamma) &= \psi(\dmove\mid \chi_{1},\dots,\chi_{\m} ).
 \end{align*}
 \end{definition}

The theorem below implies Theorem~\ref{t.angel-wins}.

 \begin{theorem}
If $(\phi_{\k},\J_{\k})$ is an implementation map for each $\k$
then the plain strategy $\eta$ defined above is a winning strategy for the
angel.
 \end{theorem}
 \begin{proof}
The plain strategy of $\eta$ was defined above via the nested strategy 
$\psi$.
Tracing back the definition of $\psi$ we see that
the next move $\amove_{i}$ of the angel is always computed
applying $\phi_{1}(\dmove_{i}\mid \record)$ for an appropriate record:
see~\eqref{e.level-1} and~\eqref{e.halting-case}.
Since $\phi_{1}$ is an implementation map, the resulting move is always
allowed.
 \end{proof}

\section{Conclusions}

One would think that a strategy depending only on the present position and
measure should also be possible.

\section{Acknowledgement}

The author acknowledges the comments of David Charlton during the many
discussions of the problem.

\end{document}